\documentclass[12pt]{amsart}
\usepackage{a4wide,enumerate,xcolor,graphicx}
\usepackage{amsmath,comment}
\usepackage{scrextend} 
\allowdisplaybreaks

\let\pa\partial
\let\na\nabla
\let\eps\varepsilon
\newcommand{\N}{{\mathbb N}}
\newcommand{\R}{{\mathbb R}}
\newcommand{\diver}{\operatorname{div}}
\newcommand{\T}{{\mathcal T}}
\newcommand{\E}{{\mathcal E}}
\newcommand{\m}{\operatorname{m}}
\newcommand{\Eint}{{\mathcal E}_{{\rm int},K}}
\newcommand{\dist}{{\operatorname{d}}}

\newtheorem{theorem}{Theorem}
\newtheorem{lemma}[theorem]{Lemma}

\newtheorem{proposition}[theorem]{Proposition}
\newtheorem{remark}[theorem]{Remark}

\begin{document}

\title[Discrete boundedness-by-entropy method]{A discrete boundedness-by-entropy
method \\ for finite-volume approximations \\ of cross-diffusion systems}

\author[A. J\"ungel]{Ansgar J\"ungel}
\address{Institute for Analysis and Scientific Computing, Vienna University of
	Technology, Wiedner Hauptstra\ss e 8--10, 1040 Wien, Austria}
\email{juengel@tuwien.ac.at}

\author[A. Zurek]{Antoine Zurek}
\address{Institute for Analysis and Scientific Computing, Vienna University of
	Technology, Wiedner Hauptstra\ss e 8--10, 1040 Wien, Austria}
\email{antoine.zurek@tuwien.ac.at}

\date{\today}

\thanks{The authors have been partially supported by the Austrian Science Fund (FWF), 
grants P30000, P33010, F65, and W1245, and by the multilateral project of the
Austrian Agency for International Cooperation in Education and Research
(OeAD), grant MULT 11/2020. This work received funding from the European 
Research Council (ERC) under the European Union's Horizon 2020 research and 
innovation programme, ERC Advanced Grant NEUROMORPH, no.~101018153.}

\begin{abstract}
An implicit Euler finite-volume scheme for general cross-diffusion systems 
with volume-filling constraints is proposed and analyzed. 
The diffusion matrix may be nonsymmetric and
not positive semidefinite, but the diffusion system is assumed to possess a
formal gradient-flow structure that yields $L^\infty$ bounds on the continuous
level. Examples include the Maxwell--Stefan systems for gas mixtures, tumor-growth
models, and systems for the fabrication of thin-film solar cells.
The proposed numerical scheme preserves the structure of the continuous
equations, namely the entropy dissipation inequality as well as the 
nonnegativity of the concentrations and the volume-filling constraints. 
The discrete entropy
structure is a consequence of a new vector-valued discrete chain rule.
The existence of discrete solutions, their positivity, and the convergence of
the scheme is proved. The numerical scheme is implemented for a 
one-dimensional Maxwell--Stefan model and a two-dimensional thin-film 
solar cell system. It is illustrated that the convergence 
rate in space is of order two and the discrete relative entropy decays
exponentially.
\end{abstract}



\maketitle


\section{Introduction}

Cross-diffusion is a phenomenon in multi-species systems, in which the gradient 
of the concentration of one species induces a flux of the other species.
Examples include gas mixtures, ion transport through membranes, and 
tumor-growth models.
Mathematically, cross-diffusion is described by quasilinear parabolic equations
with a non-diagonal diffusion matrix. The analysis of such systems is challenging since
the diffusion matrix is generally neither symmetric nor positive semidefinite,
and standard tools like maximum principles and regularity theory generally
do not apply. In recent years, it has been found
that a class of cross-diffusion systems, which describe volume-filling effects 
in mixtures, possess global bounded weak solutions
\cite{BDPS10,Jue15}. The existence proof is based on the formal gradient-flow
or entropy structure of the cross-diffusion equations, leading to the so-called 
boundedness-by-entropy method. In this paper,
we develop a discrete version of this method for finite-volume approximations
of cross-diffusion systems, preserving the structure of the continuous equations,
namely nonnegativity, boundedness, mass control, and entropy dissipation,
and converging to the continuous equations when the mesh parameters tend to zero.

\subsection{The boundedness-by-entropy method}

We consider the cross-diffusion system
\begin{equation}\label{1.eq}
  \pa_t u_i + \diver\bigg(-\sum_{j=1}^n A_{ij}(u)\na u_j\bigg) = 0,
  \quad\mbox{in }\Omega,\ t>0,\ i=1,\ldots,n,
\end{equation}
where $u=(u_1,\ldots,u_n)$ is the vector of volume or mass fractions
and $\Omega\subset\R^2$ is a bounded domain. 
The volume or mass fraction of the solvent $u_0$
is defined by $u_0=1-\sum_{i=1}^n u_i$ such that the identity
$\sum_{i=0}^n u_i=1$ is fulfilled.
These equations describe the evolution of fluid mixtures or multicomponent systems
\cite{Jue16}. We prescribe no-flux boundary and initial conditions:
\begin{equation}\label{1.bic}
  \sum_{j=1}^n A_{ij}(u)\na u_j\cdot\nu = 0\quad\mbox{on }\pa\Omega,\ t>0, \quad
	u_i(0)=u_i^0\quad\mbox{in }\Omega,
\end{equation}
where $\nu$ denotes the exterior unit normal vector to $\pa\Omega$. 
The diffusion matrix $A(u)=(A_{ij}(u))$ is generally neither symmetric nor
positive definite obstructing the use of standard techniques.
This problem can be overcome when the system possesses a
formal gradient-flow or entropy structure. In the following, we explain
this structure. 

First, we introduce the open simplex 
$$
  D = \bigg\{u=(u_1,\ldots,u_n)\in(0,1)^n:\sum_{i=1}^n u_i<1\bigg\}
$$ 
and let a strictly convex function $h\in C^2(D;[0,\infty))$
with $h(u)=\sum_{i=0}^n h_i(u_i)$ be given. Note that $u_0$ is a function
of $u=(u_1,\ldots,u_n)$, so $h$ depends only on $(u_1,\ldots,u_n)$.
Furthermore, we set ${\mathcal H}[u]=\int_\Omega h(u)dx$. Choosing $h'(u)$
as a test function in the weak formulation of \eqref{1.eq}, a formal computation
gives 
\begin{equation}\label{1.ei}
  \frac{d{\mathcal H}}{dt} 
	+ \int_\Omega \na u:h''(u)A(u)\na u dx =0, \quad 0<t<T,
\end{equation}
where $h''(u)$ is the Hessian of $h$, and ``:'' is the Frobenius matrix product. 
If $h''(u)A(u)$ is positive (semi-) definite,
we call $h$ an entropy density, ${\mathcal H}$ an entropy,
and \eqref{1.ei} an entropy (dissipation) inequality. In many applications, 
there exist $c_A>0$ and $0<s\le 1$ such that for $z\in\R^n$ and $u\in D$,
$$
  z^\top h''(u)A(u)z \ge c_A\sum_{i=1}^n u_i^{2(s-1)}z_i^2.
$$
This means that $h''(u)A(u)$ is positive definite but possibly involving a
singularity at $u_i=0$. We refer to \cite{Jue15} and Section \ref{sec.exam} 
for some examples. 
In this situation, \eqref{1.ei} provides an $L^2(\Omega)$ estimate for $\na u_i^s$. 
Moreover, if $h':D\to\R^n$ is invertible, we conclude an $L^\infty(\Omega)$
bound for $u_i$. Indeed, the strategy of the existence analysis 
is to solve \eqref{1.eq} in the entropy variable $w=h'(u)$ 
and to define the volume fractions a posteriori via $u=(h')^{-1}(w)$. Since
$(h')^{-1}:\R^n\to D$, it holds that 
$u(x,t)\in D$, which gives the desired $L^\infty(\Omega)$ bound. 
With these tools, the global existence of bounded weak solutions 
to \eqref{1.eq}--\eqref{1.bic} can be proved \cite{Jue15}. 

The existence result may be surprising in view of the fact that the diffusion
matrix in \eqref{1.eq} may be not positive (semi-) definite, but it
can be understood by observing that
the positive definiteness of $h''(u)A(u)$ implies that equations \eqref{1.eq}
are parabolic in the sense of Petrovskii \cite[Remark 4.3]{Jue16}.
Another explanation is that equations \eqref{1.eq} can be written equivalently as
\begin{equation}\label{1.B}
  \pa_t u_i(w) - \diver\bigg(\sum_{j=1}^n B_{ij}(w)\na w\bigg) = 0,
	\quad i=1,\ldots,n,
\end{equation}
where the so-called Onsager matrix
$B=(B_{ij})$, defined by $B(w)=A(u(w))h''(u(w))^{-1}$, is positive (semi-) definite.
The task is to ``translate'' this strategy to a finite-volume setting.


\subsection{Key ideas}

The derivation of the entropy inequality \eqref{1.ei} is based on the chain rule
$h''(u)\na u=\na h'(u)$. To formulate a discrete version,
we assume that $\Omega$ is the union of cells $K$
and let $\sigma=K|L$ be the edge between two neighboring cells $K$ and $L$ 
inside of $\Omega$.
The discrete volume fraction $u_i$ is constant on each cell, and we write $u_{i,K}$
for its value. The value on the edge is denoted by $u_{i,\sigma}$. 

When the entropy density equals the sum $h(u)=\sum_{i=1}^n h_i(u_i)$, the Hessian
$h''(u)$ is diagonal, and the discrete chain rule can be formulated componentwise as
\begin{equation}\label{1.mean}
  h_i''(u_{i,\sigma})(u_{i,K}-u_{i,L}) = h_i'(u_{i,K})-h_i'(u_{i,L})
	\quad\mbox{for }\sigma=K|L.
\end{equation}
If $h_i''$ is strictly monotone, 
there exists a unique solution $u_{i,\sigma}$
to \eqref{1.mean} by the mean-value theorem.
In the case of the Boltzmann entropy $h_i(u_i)=u_i(\log u_i-1)+1$, this leads
to the logarithmic mean 
\begin{equation}\label{1.logmean}
  u_{i,\sigma} = \frac{u_{i,K}-u_{i,L}}{\log u_{i,K}-\log u_{i,L}},
\end{equation}
which has been used to develop entropy-conservative schemes for hyperbolic 
conservation laws \cite{IsRo09} and entropy-dissipative schemes for drift-diffusion
equations \cite{Bes12}. 

In the present case, we assume that $h(u)=\sum_{i=0}^n h_i(u_i)$.
Unfortunately, the Hessian $h''(u)$ is not diagonal,
since $\pa^2 h/(\pa u_i\pa u_j)=\delta_{ij}h_i''(u_i)+h_0''(u_0)$.
Then the vector-valued mean-value theorem does not allow us to determine
$u_{i,\sigma}$ like in \eqref{1.mean}. 
We overcome this issue by introducing two ideas.

Our {\em first idea} is to define $u_{0,K}=1-\sum_{i=1}^n u_{i,K}$ on the cells
but to define $u_{0,\sigma}$ (as well as $u_{i,\sigma}$ for $i=1,\ldots,n$)
from \eqref{1.mean}. Thus, in general, 
$u_{0,\sigma}\neq 1-\sum_{i=1}^n u_{i,\sigma}$. 
We set $u_K=(u_{1,K},\ldots,u_{K,n})$,
$u_\sigma=(u_{0,\sigma},\ldots,u_{n,\sigma})$, and 
$H_{ij}(u_\sigma) = \delta_{ij}h_i''(u_{i,\sigma})+h_0''(u_{0,\sigma})$
for $i,j=1,\ldots,n$.
The matrix $H(u_\sigma)=(H_{ij}(u_\sigma))$ is similar to the Hessian 
$h''$ with the exception that we use
$u_{0,\sigma}$ as the argument of $h_0''$ and not
$1-\sum_{i=1}^n u_{i,\sigma}$. 
This means that $H$ depends on all variables $u_{0,\sigma},\ldots,u_{n,\sigma}$,
while $h$ is a function of $u_{1,K},\ldots,u_{n,K}$.
We prove in Lemma \ref{lem.chain} that
\begin{equation}\label{1.chain}
  \sum_{j=1}^n H_{ij}(u_\sigma)(u_{j,K}-u_{j,L}) = (h'(u_K)-h'(u_L))_i, \quad
	i=1,\ldots,n,
\end{equation}
holds, which is the desired discrete chain rule. 

Furthermore, we need the positive (semi-) definiteness of
$h''(u)A(u)$ at $u_\sigma$. Since we have replaced $h''$ by the matrix
$H(u_\sigma)$, which contains the new variable $u_{0,\sigma}$, we cannot
evaluate $A(u)$ at $(u_{1,\sigma},\ldots,u_{n,\sigma})$. 
Instead, our {\em second idea}
is to interpret the diffusion matrix $A$ as a function of 
$u_\sigma=(u_{0,\sigma},\ldots,u_{n,\sigma})$, 
called $A_\sigma$, and to impose the positive definiteness condition
\begin{equation}\label{1.posdef}
  z^\top H(u_\sigma)A_\sigma(u_\sigma)z \ge c_A\sum_{i=1}^n u_{i,\sigma}^{2(s-1)}z_i^2
	\quad\mbox{for all }z\in\R^n.
\end{equation}


\subsection{An illustrative example}\label{sec.illu.ex}

Let us explain the second idea on a simple example with the diffusion matrix
\begin{equation}\label{1.exam}
  A = \begin{pmatrix}
	1-u_1 & -u_1 \\ -u_2 & 1-u_2
	\end{pmatrix}\quad\mbox{for }u=(u_1,u_2)\in D.
\end{equation}
Equations \eqref{1.eq} with this diffusion matrix can be formally derived
in the diffusion limit from the Euler equations with friction forces
\cite[Example 4.3]{Jue16}. 
We present in Section \ref{sec.exam} further examples.
We choose the entropy density
$$
  h(u) = \sum_{i=0}^2 u_i(\log u_i-1)+3, \quad\mbox{where }u_0=1-u_1-u_2,
$$
and compute
$$
  h''(u) = \begin{pmatrix}
	1/u_1+1/u_0 & 1/u_0 \\ 1/u_0 & 1/u_2+1/u_0
	\end{pmatrix}, \quad
	h''(u)A(u) = \begin{pmatrix}
	1/u_1 & 0 \\ 0 & 1/u_2
	\end{pmatrix}.
$$
This shows that $h''(u)A(u)$ is positive definite with
$z^\top h''(u)A(u)z=z_1^2/u_1+z_2^2/u_2$ for all $u\in D$ and $z\in\R^2$.

For the numerical approximation, we write the diffusion matrix as
$$
  \widetilde{A}(u_0,u) = \frac{1}{a(u)}\begin{pmatrix}
  u_0+u_2 & -u_1 \\ -u_2 & u_0+u_1
  \end{pmatrix}, \quad\mbox{where } a(u)=u_0+u_1+u_2.
$$
Of course, this matrix and \eqref{1.exam} coincide if the identity $u_0+u_1+u_2=1$
holds. In the numerical scheme, we do not impose this condition on the edges.
Instead, we define
$$
  H(u_\sigma) = \begin{pmatrix}
	1/u_{1,\sigma}+1/u_{0,\sigma} & 1/u_{0,\sigma} \\
	1/u_{0,\sigma} & 1/u_{2,\sigma}+1/u_{0,\sigma}
	\end{pmatrix}
$$
and $A_\sigma(u_\sigma):=\widetilde{A}(u_\sigma)$, where $u_{i,\sigma}$ 
for $i=0,1,2$ is given by \eqref{1.logmean}. 
Although $u_{0,\sigma}+u_{1,\sigma}+u_{2,\sigma}\neq 1$ is not guaranteed, 
we find that
$$
  H(u_\sigma)A_\sigma(u_\sigma) = \begin{pmatrix}
	1/u_{1,\sigma} & 0 \\ 0 & 1/u_{2,\sigma}
	\end{pmatrix},
$$
and \eqref{1.posdef} is satisfied for $s=1/2$. Note the factor $a(u_\sigma)$
in the definition of $A_\sigma$ is crucial for this result and that 
$u_{i,\sigma}>0$ for all $i$ cannot be guaranteed in general. However, we prove that
$0<a(u_\sigma)\le 1$ holds in the present case; 
see the paragraph after Theorem \ref{thm.ex}.

When $1/a(u_\sigma)>1$, this factor may be interpreted as an artificial
diffusion coefficient. Yet, even if in general $a(u_\sigma) \neq 1$, as $a(u_\sigma)$ is the sum of the logarithmic mean between $u_{i,K}$ and $u_{i,L}$ 
of {\em each} species, we always observe in our numerical experiments (not shown 
here) that $a(u_\sigma)$ is of order one. Moreover,
the solution is very close to that one obtained from a finite-volume scheme with arithmetic mean (which preserves the volume--filling constraint) instead of the logarithmic mean. Thus, as expected, this factor does {\em not} lead to over-diffusive results. Let us notice that the situation may be different if, for instance, $a(u_\sigma)$ was only given by $u_{1,\sigma}$ or $u_{2,\sigma}$ (or even $u_{1,\sigma}+u_{2,\sigma}$). In this case, we can build initial data such that $1/a(u_{\sigma})$ is much larger than one at least for the first time steps, leading to over-diffusive solutions.


\subsection{State of the art}

First existence results for cross-diffusion systems were stated under 
restrictive conditions on the
nonlinearities \cite{LSU68}. Amann \cite{Ama89} showed that
weak solutions to strongly coupled parabolic systems exist globally
if their $W^{1,p}$ norm with $p>d$ can be controlled. 
Alt and Luckhaus \cite{AlLu83} proved global existence results for systems of 
the form \eqref{1.B} with uniformly positive definite Onsager matrices.
Global {\em bounded} weak solutions were shown for a special cross-diffusion
system with volume-filling effects by Burger et al.\ \cite{BDPS10} and
later for a general class of systems in \cite{Jue15}, based on the underlying
entropy structure. Such systems arise naturally in the modeling of gas mixtures 
and in multi-species population dynamics \cite{Jue16,ZaJu17}. 

Structure-preserving finite-volume-type schemes were first designed for 
hyperbolic conservation laws fulfilling entropy stability or entropy conservation 
\cite{Tad87}. The application to cross-diffusion systems is more recent.
A convergence study of a finite-volume approximation for a nondegenerate
cross-diffusion problem was carried out in \cite{ABR11}, based on classical
quadratic energy estimates. Finite-volume approximations that satisfy a discrete
entropy inequality were suggested and analyzed in 
\cite{CCGJ19,CEM20,CaGa20,DJZ20,JuZu19,JuZu20}, while
finite-volume schemes for cross-diffusion systems preserving the volume-filling
constraints were developed in \cite{CCGJ19,CEM20,DJZ20,IbSa14}.  

The preservation of the entropy structure is achieved by designing a discrete
chain rule. In the literature, the elementary inequality 
$(u-v)(\log u-\log v)\ge 4(\sqrt{u}-\sqrt{v})^2$ is used as 
a discrete version of the chain rule $\na u\cdot\na\log u=4|\na\sqrt{u}|^2$
\cite{CCGJ19,Gli11} and the logarithmic mean \eqref{1.logmean} 
as a discrete version of the chain rule $u\na\log u=\na u$
\cite{Bes12,CaGa20,GaGa05}. 
The more general chain rule \eqref{1.mean} was suggested in
our previous work \cite{JuZu20}. 
In all these examples, the discrete chain rule is defined componentwise. 
 
In the discrete gradient method for differential equations,
related discrete chain rules are formulated to achieve energy conservation
or entropy dissipation; see the review \cite{MQR99}.
Examples are given by the Gonzales scheme \cite{Gon96} and the 
mean-value discrete gradient \cite{HLV83}. The latter technique was extended to the
average vector field method \cite{CGMM12}, which uses an average of
the differential operator $\diver(B\na w)$ and is based on the vector-valued 
mean-value theorem. However, it seems to be difficult to extract gradient estimates
from these discrete gradients.
Here, we consider for the first time (up to our knowledge)
a vector-valued chain rule leading to gradient estimates.

Let us also mention related approaches for diffusion problems.
For finite-difference schemes, 
entropy-stable and entropy-dissipative discretizations were developed
in, e.g., \cite{JePa17,KRS16,LiYu14}, extending Tadmor's framework or
using upwind approximations. When the equations possess a variational
structure involving the variational derivative of the energy/entropy,
discrete variational derivatives were defined in \cite{FuMa11}.
This approach was extended to fourth-order parabolic equations \cite{Buk20},
but it seems not to cover cross-diffusion systems. 
Energy-dissipative schemes for scalar equations
were developed also for higher-order time
integrations; see, e.g., \cite{HaLu14} for Runge--Kutta methods and
\cite{JuMi15} for one-leg multistep methods. Unfortunately, these techniques
cannot be easily adapted to our setting. Interesting
approaches are the discontinuous Galerkin time discretization of \cite{Egg19},
whose use in finite-volume schemes has still to be explored, and the
space-time Galerkin approach of \cite{BPS20}, which needs a regularizing term.

The paper is organized as follows. The numerical scheme and our main results 
(existence of discrete solutions, positivity, and convergence of the scheme) 
are introduced in Section \ref{sec.main}. We present some examples of
cross-diffusion models that satisfy our main assumptions in Section \ref{sec.exam}.
In Section \ref{sec.ex}, the existence of solutions is proved, while the convergence
of the scheme is shown in Section \ref{sec.conv}. Finally, some numerical
examples are presented in Section \ref{sec.num}.


\section{Numerical scheme and main results}\label{sec.main}

\subsection{Notation and definitions}

Let $\Omega\subset\R^2$ be a bounded, polygonal domain.
We consider only two-dimensional domains, but the generalization
to higher space dimensions is straightforward.
An admissible mesh of $\Omega$ is given by (i) a family $\T$
of open polygonal control volumes (or cells), (ii) a family $\E$ of edges, and
(iii) a family ${\mathcal P}$ of points $(x_K)_{K\in\T}$ associated to the
control volumes and satisfying Definition 9.1 in \cite{EGH00}. This definition
implies that the straight line $\overline{x_Kx_L}$ between two centers of
neighboring cells is orthogonal to the edge $\sigma=K|L$ between two cells.
For instance, Vorono\"{\i} meshes satisfy this condition \cite[Example 9.2]{EGH00}.
The size of the mesh is denoted by $\Delta x = \max_{K\in\T}\operatorname{diam}(K)$.
The family of edges $\E$ is assumed to consist of interior edges $\E_{\rm int}$
satisfying $\sigma\subset\Omega$ and boundary edges $\sigma\in\E_{\rm ext}$ satisfying
$\sigma\subset\pa\Omega$. For given $K\in\T$, $\E_K$ is the set of edges of $K$,
and it splits into $\E_K=\Eint\cup\E_{{\rm ext},K}$. 
For any $\sigma\in\E$, there exists at least one cell $K\in\T$ such that 
$\sigma\in\E_K$.

We need the following definitions. For $\sigma\in\E$, we introduce the distance
$$
  \dist_\sigma = \begin{cases}
	\dist(x_K,x_L) &\quad\mbox{if }\sigma=K|L\in\E_{{\rm int},K}, \\
	\dist(x_K,\sigma) &\quad\mbox{if }\sigma\in\E_{{\rm ext},K},
	\end{cases}
$$
where d is the Euclidean distance in $\R^2$, and the transmissibility coefficient
\begin{equation}\label{2.trans}
  \tau_\sigma = \frac{\m(\sigma)}{\dist_\sigma},
\end{equation}
where $\m(\sigma)$ denotes the Lebesgue measure of $\sigma$.
The mesh is assumed to satisfy the following regularity assumption: There exists
$\zeta>0$ such that for all $K\in\T$ and $\sigma\in\E_K$,
\begin{equation}\label{2.dd}
  \dist(x_K,\sigma)\ge \zeta \dist_\sigma.
\end{equation}

Let $T>0$, $N_T\in\N$ and introduce the time step size
$\Delta t=T/N_T$ as well as the time steps $t_k=k\Delta t$ for $k=0,\ldots,N_T$. 
We denote by ${\mathcal D}$
the admissible space-time discretization of $\Omega_T=\Omega\times(0,T)$ 
composed of an admissible mesh ${\mathcal T}$ and the values $(\Delta t,N_T)$. 

Next, we introduce the functional spaces.
The space of piecewise constant functions is defined by 
$$
  V_\T = \bigg\{v: \Omega\to\R:\exists (v_K)_{K\in\T}\subset\R,\
	v(x)=\sum_{K\in\T}v_K\mathbf{1}_K(x)\bigg\},
$$
where $\mathbf{1}_K$ is the characteristic function on $K$.
In order to define a norm on this space, we first introduce the notation
$$
  v_{K,\sigma} = \begin{cases}
	v_L &\quad\mbox{if }\sigma=K|L\in\Eint, \\
	v_K &\quad\mbox{if }\sigma\in\E_{{\rm ext},K},
	\end{cases} 
$$
for $K\in\T$, $\sigma\in\E_K$ and the discrete operators
$$
  \textrm{D}_{K,\sigma} v := v_{K,\sigma}-v_K, \quad 
	\textrm{D}_\sigma v := |\mathrm{D}_{K,\sigma} v|.
$$

The (squared) $L^2$ norm, the discrete 
$H^1$ seminorm, and the discrete $H^1$ norm on $V_\T$ are given by, respectively,
\begin{align*}
  \|v\|_{0,2,\T}^2 &= \sum_{K\in\T}\m(K)|v_K|^2, \\
  |v|_{1,2,\T}^2 &= \sum_{\sigma\in\E}\tau_\sigma|\text{D}_{\sigma} v|^2, \\
	\|v\|_{1,2,\T}^2 &= |v|^2_{1,2,\T} + \|v\|^2_{0,2,\T}.
\end{align*}
We associate to these norms a dual norm with respect to the
$L^2$ inner product,
$$
  \|v\|_{-1,2,\T} = \sup\bigg\{\int_\Omega vw dx: w\in V_\T,\
	\|w\|_{1,2,\T}=1\bigg\}.
$$
It holds that
$$
  \bigg|\int_\Omega vw dx\bigg| \le \|v\|_{-1,2,\T}\|w\|_{1,2,\T}
	\quad\mbox{for }v,w\in V_\T.
$$

Finally, we introduce the space $V_{\T,\Delta t}$ of piecewise constant
functions with values in $V_\T$,
$$
  V_{\T,\Delta t} = \bigg\{ v:\overline\Omega\times[0,T]\to\R:
	\exists(v^k)_{k=1,\ldots,N_T}\subset V_\T,\
	v(x,t) = \sum_{k=1}^{N_T} v^k(x)\mathbf{1}_{(t_{k-1},t_k]}(t)\bigg\},
$$
equipped with the discrete $L^2(0,T;H^1(\Omega))$ norm
$$
  \left(\sum_{k=1}^{N_T}\Delta t \|v^k\|_{1,2,\T}^{2}\right)^{1/2} \quad 
	\mbox{for all }v\in V_{\T,\Delta t}.
$$ 


\subsection{Numerical scheme}

We define the finite-volume scheme for the cross-diffusion model \eqref{1.eq} 
and \eqref{1.bic}. We first approximate the initial functions by
\begin{equation}\label{2.init}
  u_{i,K}^0 = \frac{1}{\m(K)}\int_K u_i^0(x)dx \quad\mbox{for }K\in\T,\ 
	i=0,\ldots,n.
\end{equation}
Let $u_K^{k-1}=(u_{1,K}^{k-1},\ldots,u_{n,K}^{k-1})$ and 
$u_{0,K}^{k-1} = 1-\sum_{i=1}^n u_{i,K}^{k-1}$
be given for $K \in \T$. Then the values $u_{i,K}^k$
are determined by the implicit Euler finite-volume scheme
\begin{align}\label{2.fvm}
  & \m(K)\frac{u_{i,K}^k-u_{i,K}^{k-1}}{\Delta t} 
	+ \sum_{\sigma\in\E_K}\mathcal{F}_{i,K,\sigma}^k = 0, \\
  \label{2.flux}
  & \mathcal{F}_{i,K,\sigma}^k = - \sum_{j=1}^n \tau_\sigma 
	A_{ij,\sigma}(u^k_\sigma) 
	\textrm{D}_{K,\sigma} u^k_j \quad\mbox{for }K\in \T,\ \sigma \in \E_K,
\end{align}
and $\tau_\sigma$ is defined by \eqref{2.trans}. 
The matrix $A_\sigma=(A_{ij,\sigma})$ satisfies $A_\sigma(u_0,u)=A(u)$ for all
$u\in D$ and $u_0=1-\sum_{i=1}^n u_i$. By definition of the
discrete operator $\textrm{D}_{K,\sigma}$, the discrete fluxes vanish on
the boundary edges, guaranteeing the no-flux boundary conditions. Thus, 
we only need to define in \eqref{2.flux} the mean vector $u^k_\sigma = (u^k_{0,\sigma},
\ldots,u^k_{n,\sigma})$ for every $\sigma =K|L \in \E_{\rm int}$:
\begin{equation}\label{2.usigma}
  u^k_{i,\sigma} = \begin{cases}
  \widetilde u_{i,\sigma}^k
	\quad & \mbox{if } u^k_{i,K} > 0, \ u^k_{i,L} > 0, \mbox{ and } 
	u^k_{i,K} \neq u^k_{i,L}, \\
  u^k_{i,K} & \mbox{if } u^k_{i,K} = u^k_{i,L} > 0, \\
  0 & \mbox{else},
  \end{cases}   
\end{equation}
where $\widetilde u_{i,K}^k\in(0,1)$ is the unique solution to
\begin{equation}\label{2.chain}
  h_i''(\widetilde u_{i,\sigma}^k)\textrm{D}_{K,\sigma}u_i^k
	= \textrm{D}_{K,\sigma} h_i'(u_i^k)\quad\mbox{for }K\in\T,\ \sigma\in
	\E_{\mathrm{int},K},\	i=0,\ldots,n.
\end{equation}
If $h''_i$ is continuous and strictly monotone, the existence of a unique value 
$\widetilde u_{i,\sigma}^k$ follows from the mean-value theorem. 
Moreover, if $u^k_{i,K}$, $u^k_{i,L} \ge 0$ we have for $i=0,\ldots,n$,
$$
  0 \leq \min\{u_{i,K}^k,u_{i,L}^k\} \leq u_{i,\sigma}^k
	\leq \max\{u_{i,K}^k,u_{i,L}^k\}\le 1.
$$


\subsection{Main results}\label{ssec.main}

Given the entropy density $h(u)=\sum_{i=1}^n h_i(u_i)+h_0(u_0)$, 
we define for $u\in V_{\T}$ the discrete entropy
$$
  {\mathcal H}[u] = \sum_{K\in\T}\m(K)h(u_K),
$$
and replace the Hessian $h''$ by the matrix
\begin{equation}\label{2.H}
  H_{ij}(u_\sigma) = \delta_{ij}h_i''(u_{i,\sigma}) + h_0''(u_{0,\sigma}), 
	\quad u_\sigma\in(0,1)^{n+1},\ i,j=1,\ldots,n,
\end{equation}
where $u_{i,\sigma}$ is defined by \eqref{2.usigma}.
Note that this matrix is symmetric and positive definite (if $h_i$ is strictly convex).

We impose the following hypotheses:

\begin{labeling}{(A44)}
\item[\textbf{(H1)}] Domain: $\Omega\subset\R^2$ is a bounded polygonal domain
and $D=\{u=(u_1,\ldots,u_n)\in(0,1)^n:\sum_{i=1}^n u_i<1\}$.

\item[\textbf{(H2)}] Discretization: $\mathcal{D}$ is an admissible discretization of 
$\Omega_T=\Omega\times(0,T)$ satisfying \eqref{2.dd}.

\item[\textbf{(H3)}] Initial data: $u^0=(u^0_1,\ldots,u_n^0) \in 
L^1(\Omega;D)$ satisfies $\int_\Omega h(u^0) dx < \infty$. 
We set $u_0^0=1-\sum_{i=1}^n u_i^0$.

\item[\textbf{(H4)}] Entropy density: $h(u)=\sum_{i=1}^n h_i(u_i)+h_0(u_0)$
for $u\in D$ and $u_0=1-\sum_{i=1}^n u_i$, where
$h\in C^0(\overline{D};[0;\infty))$ is convex,
$h':D\to\R$ is invertible, $h_i\in C^2(0,1)$, $h_i''$ is strictly decreasing, 
and there exists $c_h>0$ 
such that $h_i(x) \geq c_h(x-1)$ for all $0 \leq x \leq 1$, $i=1,\ldots,n$.

\item[\textbf{(H5)}] Diffusion matrix: $A \in C^{0,1}(\overline{D};\R^{n\times n})$ 
and there exists a matrix $A_\sigma\in C^{0,1}([0,1]\times (0,1)^n;\R^{n\times n})$ 
such that $A(u)=A_\sigma(u_\sigma)$ for 
all $u\in D$ with $u_i=u_{i,\sigma}$ for $i=1,\ldots,n$ and 
$u_{0,\sigma}=1-\sum_{i=1}^n u_i$. We assume that $\|A_\sigma(0,u)\|<\infty$, 
where $\|\cdot\|$ denotes some matrix norm, for all $u=(u_1,\ldots,u_n) \in (0,1)^n$ 
satisfying $\sum_{i=1}^n u_i \leq 1$, and there exist numbers $c_A>0$, $0<s< 1$
such that for all $z\in\R^n$ and for some $u_\sigma\in(0,1)^{n+1}$,
$$
	z^\top H(u_\sigma)A_\sigma(u_\sigma)z 
	\ge c_A\sum_{i=1}^n u_{i,\sigma}^{2(s-1)}z_i^2.
$$
\end{labeling}

These assumptions include all hypotheses needed in the boundedness-by-entropy
method; see \cite{Jue15}. We also need additional conditions. 
First, the entropy density in Hypothesis (H4)
has a particular structure including the Boltzmann entropy
for volume-filling models. The strict monotonicity of $h''_i$ is required 
to define properly the mean value $\widetilde u^k_{i,\sigma}$ in \eqref{2.usigma}.
Admissible examples are $h_i(s)=s(\log s-1)+1$ and, more generally, 
$h_i(s)=\int_a^s\log q(z)dz$ with $a\in(0,1)$, $q\in C^2(0,1)\cap C^0([0,1])$ 
is strictly monotone, $qq''>(q')^2$, $q(0)=0$, and $q(0)/q'(0)=0$.

Second, the positive definiteness condition in Hypothesis (H5)
is formulated for the matrix
$H(u_\sigma)$, which replaces the Hessian $h''$, and the modified diffusion matrix
$A_\sigma(u_\sigma)$. This modification is needed to take care of the fact
that $u_{0,\sigma}$ can generally not be 
identified with $1-\sum_{i=1}^n u_{i,\sigma}$.
If this identification is possible, the matrices $h''A$ and $HA_\sigma$ coincide.
The Lipschitz continuity of $A_{\sigma}$ is needed in the proof of the convergence
of the scheme but not for the existence analysis. 
We prove below (see Theorem \ref{thm.ex}) that $u^k_{i,\sigma}>0$ holds
for all $i=1,\ldots,n$ but only $u_{0,\sigma}^k\ge 0$. 
Our analysis can be extended under suitable assumptions to the case $s=1$ and 
allowing for source terms in \eqref{1.eq}; see Remarks \ref{rem.s1}, \ref{rem.source},
and \ref{rem.conv}.

The first main result is as follows.

\begin{theorem}[Existence of discrete solutions]\label{thm.ex}
Let Hypotheses (H1)--(H5) hold. Then there exists a solution 
$u^k_K=(u^k_{1,K},\ldots,u^k_{n,K})$ to scheme \eqref{2.init}--\eqref{2.usigma} 
satisfying $u^k_K \in D$ for $K \in \T$
and $0 < u^k_{i,\sigma} < 1$ for $\sigma \in \E_{\rm int}$, $k \geq 1$, 
$i=1,\ldots,n$. Moreover, the following discrete entropy inequality holds:
\begin{equation}\label{2.ei}
  \mathcal{H}[u^k] 	+ c_A\Delta t\sum_{i=1}^n\sum_{\sigma\in\E_{\rm{int}}}\tau_\sigma 
	u_{i,\sigma}^{2(s-1)}(\mathrm{D}_\sigma u_i^k)^2 
	\le \mathcal{H}[u^{k-1}].
\end{equation}
\end{theorem}

The proof of Theorem \ref{thm.ex} is based on a topological degree argument
and the entropy estimate \eqref{2.ei}, which follows from the discrete chain rule
\eqref{1.chain}. Generally, we cannot exclude that $\sum_{i=0}^n u_{i,\sigma}^k>1$. 
However, when the entropy is given by the Boltzmann entropy 
$h_i(u_i)=u_i(\log u_i-1)+1$, it holds that $\sum_{i=0}^n u_{i,\sigma}^k\le 1$,
since the logarithmic mean is not larger than the arithmetic mean, i.e.\ 
$u_{i,\sigma}^k\le (u_{i,K}^k+u_{i,L}^k)/2$ for $\sigma=K|L$ and
$\sum_{i=0}^n u_{i,\sigma}^k\le \sum_{i=0}^n(u_{i,K}^k+u_{i,L}^k)/2=1$.
This shows that the volume-filling constraints are fully satisfied.

For the convergence result, we need some notation. For $K\in\T$ and $\sigma\in\E_K$,
we define the cell $T_{K,\sigma}$ of the dual mesh:
\begin{itemize}
\item If $\sigma=K|L\in\E_{{\rm int},K}$, then $T_{K,\sigma}$ is that cell 
(``diamond'') whose
vertices are given by $x_K$, $x_L$, and the end points of the edge $\sigma$.
\item If $\sigma\in\E_{{\rm ext},K}$, then $T_{K,\sigma}$ is that cell (``triangle'')
whose vertices are given by $x_K$ and the end points of the edge $\sigma$.
\end{itemize}
The cells $T_{K,\sigma}$ define a partition of $\Omega$. It follows from the
property that the straight line $\overline{x_Kx_L}$ between two neighboring 
centers of cells is orthogonal to the edge $\sigma=K|L$ that
\begin{equation*}
  \m(\sigma)\dist(x_K,x_L) = 2\m(T_{K,\sigma}) \quad\mbox{for }
	\sigma=K|L\in \E_{\rm int}.
\end{equation*}
The approximate gradient of $v\in V_{\T,\Delta t}$ is defined by
$$
  \na^{\mathcal D} v(x,t) = \frac{\m(\sigma)}{\m(T_{K,\sigma})}
	(\mathrm{D}_{K,\sigma} v^k)\nu_{K,\sigma}
	\quad\mbox{for }x\in T_{K,\sigma},\ t\in(t_{k-1},t_{k}],
$$
where $\nu_{K,\sigma}$ is the unit vector that is normal to $\sigma$ and points
outwards of $K$.

We introduce a family $(\mathcal{D}_m)_{m\in\N}$ of admissible space-time 
discretizations of $\Omega_T$ indexed by the size 
$\eta_m=\max\{\Delta x_m,\Delta t_m\}$ of the mesh, satisfying $\eta_m\to 0$
as $m\to\infty$. We denote by $\T_m$ the corresponding meshes of $\Omega$ and
by $\Delta t_m$ the corresponding time step sizes.  
Finally, we set $\na^m:=\na^{\mathcal{D}_m}$.

\begin{theorem}[Convergence of the scheme]\label{thm.conv}
Let the assumptions of Theorem \ref{thm.ex} hold, let $(\mathcal{D}_m)_{m\in\N}$ 
be a family of admissible meshes satisfying \eqref{2.dd} uniformly in $m\in\N$. Let $(u_m)_{m\in\N}$ be a family
of finite-volume solutions to \eqref{2.init}--\eqref{2.usigma} constructed in
Theorem \ref{thm.ex}. Then there exists a function $u=(u_1,\ldots,u_n)\in
L^2(0,T;H^1(\Omega;\R^n))$ satisfying $u(x,t)\in\overline{D}$ for a.e. $(x,t)\in\Omega_T$ and for $i=1,\ldots,n$,
\begin{align*}
  u_{i,m}\to u_i &\quad\mbox{strongly in }L^p(\Omega_T), \quad 1 \leq p < \infty, \\
	\na^m u_{i,m}\rightharpoonup\na u_i &\quad\mbox{weakly in }L^2(\Omega_T)
	\quad\mbox{as }m\to\infty,
\end{align*}
up to a subsequence,
and $u$ is a weak solution to \eqref{1.eq} and \eqref{1.bic}, i.e., for all
$\psi_i\in C_0^\infty(\Omega\times[0,T))$, it holds that for all $i=1,\ldots,n$,
\begin{equation}\label{2.weak}
  \int_0^T\int_\Omega u_i\pa_t\psi_i dxdt + \int_\Omega u_i^0\psi_i(0)dx 
	= \int_0^T\int_\Omega\sum_{j=1}^n A_{ij}(u)\na u_j\cdot\na\psi_i dxdt.
\end{equation}
\end{theorem}

The proof is based on uniform estimates deduced from the entropy inequality
\eqref{2.ei} and the compactness result from \cite{GaLa12}, giving a.e.\
convergence of a subsequence of $(u_m)$. We follow the strategy of
\cite{CLP03} to show that the limit satisfies \eqref{1.eq} in the weak sense 
\eqref{2.weak}.


\section{Examples and counter-example}\label{sec.exam}

\subsection{Three-species Maxwell--Stefan equations}\label{ssec.MaxwellStefan}

The Maxwell--Stefan equations describe the evolution of the partial
densities in a multicomponent fluid,
with diffusion fluxes originating from friction forces.
The diffusion matrix of the Fock--Onsager formulation is defined for the 
three-species case by
\begin{align*}
  & A(u) = \frac{1}{\alpha(u)}\begin{pmatrix}
	d_2 + (d_0-d_2)u_1 & (d_0-d_1)u_1 \\
	(d_0-d_2)u_2 & d_1 + (d_0-d_1)u_2
	\end{pmatrix}, \\
  & \mbox{where }\alpha(u) = d_1d_2(1-u_1-u_2) + d_0d_1u_1 + d_0d_2u_2,
\end{align*}
where $u_1$, $u_2$ are the volume fractions of the components of the fluid mixture
and $d_i$ are some positive parameters \cite{Jue16}. The third constitutent is
given by $u_0=1-u_1-u_2$. We introduce the entropy density 
\begin{align}\label{3.density.Boltz}
  h(u) = u_1(\log u_1-1) + u_2(\log u_2-1) + u_0(\log u_0-1) + 3, \quad u\in D.
\end{align}
Then the Hessian of $h$ equals
$$
  h''(u) = \begin{pmatrix} 1/u_1 + 1/u_0 & 1/u_0 \\
	1/u_0 & 1/u_2 + 1/u_0 \end{pmatrix}.
$$
Now, for the finite-volume scheme, let $(u_{1,K},\ldots,u_{n,K}) \in D$ be given 
for every $K \in \T$, we define for all $\sigma = K|L \in \E_{int}$ and 
$i=0,\ldots,n$,
\begin{equation}\label{3.usigma}
  u_{i,\sigma} = \begin{cases}
  \dfrac{u_{i,K}-u_{i,L}}{\log(u_{i,K})-\log(u_{i,L})}
	\quad & \mbox{if } u_{i,K} > 0, \ u_{i,L} > 0, \mbox{ and } 
	u_{i,K} \neq u_{i,L}, \\
  u_{i,K} & \mbox{if } u_{i,K} = u_{i,L} > 0, \\
  0 & \mbox{else},
  \end{cases}   
\end{equation}
and
\begin{align*}
  & H(u_\sigma) = \begin{pmatrix} 
	1/u_{1,\sigma} + 1/u_{0,\sigma} & 1/u_{0,\sigma} \\
	1/u_{0,\sigma} & 1/u_{2,\sigma} + 1/u_{0,\sigma} \end{pmatrix}, \\
	& A_\sigma(u_\sigma) = \frac{1}{\alpha_\sigma(u_\sigma)}\begin{pmatrix}
	d_2(u_{2,\sigma}+u_{0,\sigma}) + d_0u_{1,\sigma} & (d_0-d_1)u_{1,\sigma} \\
	(d_0-d_2)u_{2,\sigma} & d_1(u_{1,\sigma}+u_{0,\sigma}) + d_0u_{2,\sigma}
	\end{pmatrix}, \\ 
	&\mbox{where }\alpha_\sigma(u_\sigma) = d_1d_2u_{0,\sigma} + d_0d_1u_{1,\sigma}
	+ d_0d_2u_{2,\sigma}.
\end{align*}
Note that $A_\sigma=A$ if $u_{0,\sigma}=1-u_{1,\sigma}-u_{2,\sigma}$.
We compute for $z=(z_1,z_2)\in\R^2$,
\begin{align*}
  H(u_\sigma)A_\sigma(u_\sigma) 
	&= \frac{u_{0,\sigma}+u_{1,\sigma}+u_{2,\sigma}}{\alpha_\sigma(u_\sigma)}
	\begin{pmatrix}
	d_2/u_{1,\sigma} + d_0/u_{0,\sigma} & d_0/u_{0,\sigma} \\
	d_0/u_{0,\sigma} & d_1/u_{2,\sigma} + 1/u_{0,\sigma}
	\end{pmatrix}, \\
  z^\top H(u_\sigma)A_\sigma(u_\sigma)z 
	&= \frac{u_{0,\sigma}+u_{1,\sigma}+u_{2,\sigma}}{\alpha_\sigma(u_\sigma)}\bigg(
	\frac{d_2}{u_{1,\sigma}}z_1^2
	+ \frac{d_1}{u_{2,\sigma}}z_2^2 + \frac{d_0}{u_{0,\sigma}}(z_1+z_2)^2\bigg) \\
	&\ge c\bigg(\frac{z_1^2}{u_{1,\sigma}} + \frac{z_2^2}{u_{2,\sigma}}\bigg),
\end{align*}
where $c=\max\{d_0d_1,d_0d_1,d_1d_2\}>0$. This fulfills Hypothesis (H5) with
$s=1/2$. Moreover, Theorem \ref{thm.ex} shows that $u_{i,\sigma}>0$ for
$i=1,2$. Thus, $\alpha_\sigma(u_\sigma)$ is positive and $A_\sigma$
is well defined.


\subsection{A cross-diffusion system for thin-film solar cells}\label{sec.exam.thinfilm}

The physical vapor deposition process for the fabrication of thin-film crystalline 
solar cells can be described by the following cross-diffusion equations:
\begin{align}\label{2.thinfilm}
  \pa_t u_i = \diver\bigg(\sum_{j=0}^n a_{ij}(u_j\na u_i-u_i\na u_j)\bigg),
	\quad i=0,\ldots,n,
\end{align}
where $u_i$ are the volume fractions of the components of the thin film
and $a_{ij}>0$ satisfies $a_{ij}=a_{ji}$ for all $i,j=1,\ldots,n$. Since 
$\sum_{i=0}^n u_i=1$, we can remove, as in \cite{BaEh18}, 
the equation for the species $i=0$,
leading to equations \eqref{1.eq} with the diffusion matrix $A(u)=(A_{ij}(u))$, 
where
\begin{align}\label{2.diff.matrix.thinfilm}
  A_{ii}(u) = \sum_{k=1,\,k\neq i}^n(a_{ik}-a_{i0})u_k + a_{i0}, \quad
	A_{ij}(u) = -(a_{ij}-a_{i0})u_i\quad\mbox{for }j\neq i,
\end{align}
and $i,j=1,\ldots,n$. In this case, we consider the entropy density 
\eqref{3.density.Boltz} and for $(u_{1,K},\ldots,u_{n,K}) \in D$ given for every 
$K \in \T$, we define for all $\sigma = K|L \in \E_{int}$ and $i=0,\ldots,n$ 
the coefficient $u_{i,\sigma}$ as in \eqref{3.usigma}. Then we choose the 
following matrices:
\begin{align}
  & H_{ii}(u_\sigma) = \frac{1}{u_{i,\sigma}} +  \frac{1}{u_{0,\sigma}}, \quad
	H_{ij}(u_\sigma) =  \frac{1}{u_{0,\sigma}}\quad\mbox{for }j\neq i, \label{2.Hsigma} \\
  & A_{ii,\sigma}(u_\sigma) = \sum_{k=1,\,k\neq i}^n(a_{ik}-a_{i0})u_{k,\sigma} 
	+ a_{i0}, \quad
	A_{ij,\sigma}(u_\sigma) = -(a_{ij}-a_{i0})u_{i,\sigma}\quad\mbox{for }j\neq i.
	\label{2.Asigma}
\end{align}

We claim that Hypothesis (H5) holds with $s=1/2$.
The proof follows the strategy in \cite[Section 3.1]{BaEh18}, but since
generally $\beta\neq 1$, we need to modify slightly the arguments.
To this end, we introduce the matrix $P(u_\sigma)$ with elements 
$P_{ij}(u_\sigma)=\delta_{ij}-u_{i,\sigma}$ for $i,j=1,\ldots,n$. 
A computation shows that
$$
  (H(u_\sigma)P(u_\sigma))_{ii} 
	= \frac{1}{u_{i,\sigma}} + \frac{1-\beta}{u_{0,\sigma}}, \quad
	(H(u_\sigma)P(u_\sigma))_{ij} 
  = \frac{1-\beta}{u_{0,\sigma}}\quad\mbox{for }i\neq j.
$$
Recall that $\beta\le 1$. Then $H(u_\sigma)P(u_\sigma)$ is positive definite with
$$
  z^\top H(u_\sigma)P(u_\sigma)z 
	= \sum_{i=1}^n\frac{z_i^2}{u_{i,\sigma}} 
	+ \frac{1-\beta}{u_{0,\sigma}}\sum_{i,j=1}^n(z_i+z_j)^2
	\ge \sum_{i=1}^n\frac{z_i^2}{u_{i,\sigma}}
$$
for any $z\in\R^n$.
We also need the matrix $\Lambda(u_\sigma)$ with elements 
$\Lambda_{ij}(u_\sigma)=\delta_{ij}/u_{i,\sigma}$ and
$\alpha=\min_{i,j=1,\ldots,n}a_{ij}>0$. The previous inequality gives
$$
  z^\top \big(H(u_\sigma)A_\sigma(u_\sigma)-\alpha\Lambda(u_\sigma)\big)z
	\ge z^\top H(u_\sigma)\big(A_\sigma(u_\sigma)-\alpha P(u_\sigma)\big)z.
$$
Denoting by $\widetilde A_\sigma(u_\sigma)$ the matrix with coefficients
$a_{ij}-\alpha$ instead of $a_{ij}$ and introducing the matrix $D(u_\sigma)$ 
with elements $D_{ij}(u_\sigma)=u_{i,\sigma}$ for $i,j=1,\ldots,n$, it follows that
$A_\sigma(u_\sigma)-\alpha P(u_\sigma)
=\widetilde A_\sigma(u_\sigma)+\alpha D(u_\sigma)$. Therefore,
$$
  z^\top \big(H(u_\sigma)A_\sigma(u_\sigma)-\alpha\Lambda(u_\sigma)\big)z
	\ge z^\top\big(H(u_\sigma)\widetilde A_\sigma(u_\sigma)
	+ \alpha H(u_\sigma)D(u_\sigma)\big)z.
$$

It remains to show that $H(u_\sigma)\widetilde A_\sigma(u_\sigma)$
and $H(u_\sigma)D(u_\sigma)$ are positive semidefinite.
All elements of $H(u_\sigma)D(u_\sigma)$ are given by the same value
$\beta/u_{0,\sigma}$, and so this matrix is positive semidefinite. 
Furthermore, $H(u_\sigma)\widetilde A_\sigma(u_\sigma)$
is positive semidefinite if and only if 
$\widetilde A_\sigma(u_\sigma)H(u_\sigma)^{-1}$ is positive semidefinite.
(At this point, we use the symmetry of $H(u_\sigma)$.)  
Since the elements of the inverse $H(u_\sigma)^{-1}$ are 
$$
  H(u_\sigma)^{-1}_{ii} = \frac{1}{\beta}(\beta-u_{i,\sigma})u_{i,\sigma}, \quad
	H(u_\sigma)^{-1}_{ij} = -\frac{1}{\beta}u_{i,\sigma}u_{j,\sigma}
	\quad\mbox{for }i\neq j
$$
and $i,j=1,\ldots,n$, we obtain
\begin{align*}
  \big(\widetilde A_\sigma(u_\sigma)H(u_\sigma)^{-1}\big)_{ii}
	&= (a_{i0}-\alpha)\frac{u_{i,\sigma}}{\beta}(\beta-u_{i,\sigma})
	+ \frac{1}{\beta}\sum_{k=1,\,j\neq i}^n(a_{ij}-\alpha)u_{i,\sigma}u_{j,\sigma}, \\
	\big(\widetilde A_\sigma(u_\sigma)H(u_\sigma)^{-1}\big)_{ij}
	&= -\frac{1}{\beta}(a_{ij}-\alpha)u_{i,\sigma}u_{j,\sigma}\quad\mbox{for }i\neq j.
\end{align*}
Consequently, using the symmetry of $(a_{ij})$,
\begin{align*}
  z^\top &\widetilde A_\sigma(u_\sigma)H(u_\sigma)^{-1}z \\
  &= \frac{1}{\beta}\sum_{i=1}^n(a_{i0}-\alpha)u_{i,\sigma}(\beta-u_{i,\sigma})z_i^2 
	+ \frac{1}{\beta}\sum_{i=1}^n\sum_{j=1,\,j\neq i}^n(a_{ij}-\alpha)
	u_{i,\sigma}u_{j,\sigma}(z_i^2 - z_iz_j) \\
	&= \frac{1}{\beta}\sum_{i=1}^n(a_{i0}-\alpha)u_{i,\sigma}(\beta-u_{i,\sigma})z_i^2
	+ \frac{1}{2\beta}\sum_{i\neq j}(a_{ij}-\alpha)
	u_{i,\sigma}u_{j,\sigma}(z_i^2+z_j^2 - 2z_iz_j) \ge 0.
\end{align*}
We conclude that
$$
  z^\top \big(H(u_\sigma)A_\sigma(u_\sigma)-\alpha\Lambda(u_\sigma)\big)z	\ge 0.
$$
Summarizing, the result reads as follows.

\begin{lemma}\label{lem.thinfilm}
et $u_\sigma \in (0,1)^{n+1}$ be defined by \eqref{3.usigma} and $H(u_\sigma)$ and 
$A_\sigma(u_\sigma)$ by \eqref{2.Hsigma}--\eqref{2.Asigma} and assume that 
$a_{ij}=a_{ji}$ for all
$i,j=1,\ldots,n$ and $\alpha=\min_{i,j=1,\ldots,n}a_{ij}>0$. 
Then for any $z\in\R^n$,
$$
  z^\top H(u_\sigma)A_\sigma(u_\sigma)z
	\ge \alpha\sum_{i=1}^n\frac{z_i^2}{u_{i,\sigma}}.
$$
\end{lemma}

Already in \cite{CaGa20}, a convergent two-point flux approximation 
finite-volume scheme for this model with a logarithmic mean was introduced,
but with a different strategy. Indeed, the authors of \cite{CaGa20} noticed that if
all diffusion coefficients are equal, system \eqref{2.thinfilm} reduces to
$n+1$ uncoupled heat equations, and they rewrite the equations as
$$
  \pa_t u_i - a^* \Delta u_i = \diver\bigg(\sum_{j=0}^n (a_{ij}-a^*)
  (u_j\na u_i-u_i\na u_j)\bigg), \quad i=0,\ldots,n,
$$
where $a^*>0$ is arbitrary. Then they designed their scheme for this equivalent
system and proved its convergence, using similar techniques as in our paper.
However, the numerical results depend on the choice of $a^*$, and choosing 
$a^*>0$ too large overestimates the diffusion. In our approach, we avoid
the artificial parameter $a^*$ but still obtain full structure preservation
of the scheme.


\subsection{Tumor-growth model}\label{sec.tumor}

The growth of an avascular tumor can be modeled in the framework of fluid
dynamics and continuum mechanics by diffusion fluxes of the tumor cells,
the extracellular matrix (ECM), and the interstitial fluid (water, nutrients).
The diffusion matrix of the tumor-growth model of \cite{JaBy02}
is given by
$$
A(u) = \begin{pmatrix}
2u_1(1-u_1) - \beta\theta u_1u_2^2 & -2\beta u_1u_2(1+\theta u_1) \\
-2u_1u_2 + \beta\theta(1-u_2)u_2^2 & 2\beta u_2(1-u_2)(1+\theta u_1)
\end{pmatrix},
$$
where $u_1$ is the volume fraction of the tumor cells and $u_2$ is the volume
fraction of the ECM. The volume fraction of the interstitial fluid is denoted
by $u_0$, and it holds that $u_0+u_1+u_2=1$.
The entropy density and the mobility coefficients are defined as in the 
previous examples, and we choose
$H(u_\sigma)$ as before. Furthermore, we define
$$
A_\sigma(u_\sigma) = \frac{1}{a(u_\sigma)}\begin{pmatrix}
2u_{1,\sigma}(u_{0,\sigma}+u_{2,\sigma}) - \beta\theta u_{1,\sigma}u_{2,\sigma}^2 
& -2\beta u_{1,\sigma}u_{2,\sigma}(1+\theta u_{1,\sigma}) \\
-2u_{1,\sigma}u_{2,\sigma} + \beta\theta(u_{0,\sigma}+u_{1,\sigma})u_{2,\sigma}^2 
& 2\beta u_{2,\sigma}(u_{0,\sigma}+u_{1,\sigma})(1+\theta u_{1,\sigma})
\end{pmatrix},
$$
where $a(u_\sigma)=u_{0,\sigma}+u_{1,\sigma}+u_{2,\sigma}$ is a correction factor.
The corresponding cross-diffusion system has an entropy structure under
the condition $\theta<4/\sqrt{\beta}$ \cite{JuSt12}. We compute
$$
z^\top H(u_\sigma)A_\sigma(u_\sigma)z
= 2z_1^2 + \beta\theta u_{2,\sigma}z_1z_2 + 2\beta(1+\theta u_{1,\sigma})z_2^2
\ge \delta(z_1^2+z_2^2),
$$
where $\delta>0$ depends on $\beta$ and $\theta$. This does not
fulfill Hypothesis (H5) since $s=1$. Moreover, we cannot deduce that $a(u_\sigma)>0$ from Theorem \ref{thm.ex}.
For instance, if $\sigma=K|L$ and $u_{1,K}=u_{2,L}=0$, $u_{1,L}=u_{2,K}=1$,
we obtain $u_{0,K}=u_{0,L}=0$ and $u_{i,\sigma}=0$ for all $i=0,1,2$. 
In fact, we observe in our numerical simulations that $a(u_\sigma)$ 
may vanish. This can be prevented by adding an artificial diffusion term of the 
form $\delta\Delta u_i$ with $\delta>0$ for $i=1,2$. Then Hypothesis (H5) is
satisfied with $s=1/2$. However, such terms regularize the solutions in such a way
that the ``spikes'' observed in  \cite[Figure 1]{JuSt12} are smoothed out.
Thus, the accurate numerical simulation in the case $s=1$ is still an open 
problem.


\section{Proof of Theorem \ref{thm.ex}}\label{sec.ex}

We first prove a discrete version of the chain rule $h''(u)\na u=\na h'(u)$.

\begin{lemma}[Discrete chain rule]\label{lem.chain}
Let $H_{ij}(u_\sigma)$ be defined by \eqref{2.H}, where $u\in D$ and $u_\sigma$ is given by \eqref{2.usigma}. Then for all $\sigma\in\Eint$, it holds that
$$
  H(u_\sigma)\mathrm{D}_{K,\sigma}u^k = \mathrm{D}_{K,\sigma}h'(u^k).
$$
\end{lemma}

\begin{proof}
Let $\sigma=K|L\in\Eint$ and $i\in\{1,\ldots,n\}$. By definition \eqref{2.usigma}, 
we find that
\begin{align*}
  \sum_{j=1}^n H_{ij}(u_\sigma^k)\textrm{D}_{K,\sigma}u_j^k 
	&= h_i''(u_{i,\sigma})(u_{i,L}-u_{i,K})
	+ h_0''(u_{0,\sigma})\sum_{j=1}^n(u_{j,L}-u_{j,K}) \\
	&= h_i''(u_{i,\sigma})(u_{i,L}-u_{i,K})
	- h_0''(u_\sigma)(u_{0,L}-u_{0,K}),
\end{align*}
using $\sum_{j=1}^n u_{i,K}=1-u_{0,K}$ in the last step.
Since $u\in D$, we have either $u_{i,\sigma}=\widetilde u_{i,\sigma}$ or $u_{i,\sigma}=u_{i,K}=u_{i,L}$.
Therefore, 
\begin{align*}
  \sum_{j=1}^n H_{ij}(u_\sigma^k)\textrm{D}_{K,\sigma}u_j^k
  &= h_i'(u_{i,L})-h_i'(u_{i,K}) - \big(h_0'(u_{0,L})-h_0'(u_{i,K})\big) \\
	&= \textrm{D}_{K,\sigma}h_i'(u^k) - \textrm{D}_{K,\sigma}h_0'(u^k)
	= \textrm{D}_{K,\sigma}(\pa h/\pa u_i)(u^k).
\end{align*}
The result also holds when $\sigma\in\E_{{\rm ext},K}$, since then
$\textrm{D}_{K,\sigma}u_j^k=0$. This finishes the proof. 
\end{proof}

The existence proof is similar to \cite[Section 2]{JuZu20} 
and we repeat only the main arguments. 
The proof of the positivity statements, however, is new.

{\em Step 1: Fixed-point problem.}
We proceed by induction over $k\in{\mathbb N}\cup\{0\}$.
If $k=0$, we have $u_{K}^0\in D$ for $K\in\T$ by Hypothesis (H3).
Let $u^{k-1}$ be given such that $u_K^{k-1}\in D$ for $K\in\T$. 
We construct $u^k$ from a fixed-point argument. For this, 
let $R>0$, $\eps>0$ and set
$$
  Z_R = \big\{w=(w_1,\ldots,w_n)\in V_\T^n:\|w_i\|_{1,2,\T}<R\mbox{ for }
	i=1,\ldots,n\big\}.
$$
We define the mapping $F_\eps:Z_R\to\R^{n\theta}$, $F_\eps(w)=w^\eps$, where
$\theta=\#\T$ and $w^\eps=(w_1^\eps,\ldots,w_n^\eps)$ solves the linear system
\begin{equation}\label{3.lin}
  \eps\sum_{\sigma\in\E_K}\tau_\sigma \text{D}_{K,\sigma}w_i^\eps
	- \eps\m(K) w^\eps_{i,K}
  = \frac{\m(K)}{\Delta t}(u_{i,K}-u_{i,K}^{k-1})
	+ \sum_{\sigma\in\E_K}\mathcal{F}_{i,K,\sigma},
\end{equation}
for $K\in\T$, $i=1,\ldots,n$, where $\mathcal{F}_{i,K,\sigma}$ is defined in
\eqref{2.flux} and $u_K=(h')^{-1}(w_K)\in D$. The regularization in $\eps$
is needed, since the diffusion matrix is only positive semidefinite in the
variable $w^\eps$. The existence of a unique solution
$w^\eps$ to \eqref{3.lin} is a consequence of the proof of \cite[Lemma 9.2]{EGH00}.
The continuity of $F_\eps$ is shown as in \cite[Section 4]{JuZu20} by
exploiting the fact that $w\in Z_R$ is bounded and so does $u=(h')^{-1}(w)$,
yielding the estimate $\eps\|w_i^\eps\|_{1,2,\T}\le C(R)$ for some constant
$C(R)>0$. 

We claim that $F_\eps$ admits a fixed point. To this end, we use a topological
degree argument \cite[Chap.~1]{Dei85} and prove that the Brouwer topological
degree satisfies $\operatorname{deg}(I-F_\eps,Z_R,0)=1$. It is sufficient to
verify that any solution $(w^\eps,\rho)\in\overline{Z}_R\times[0,1]$ to the fixed-point
equation $w^\eps=\rho F_\eps(w)$ satisfies $(w_\eps,\rho)\not\in\pa Z_R\times[0,1]$ for
sufficiently large values of $R>0$. Let $(w^\eps,\rho)$ be a fixed point
and $\rho\neq 0$ (the case $\rho=0$ is clear). Then $w^\eps$ solves
\begin{equation}\label{3.approx}
  \eps\sum_{\sigma\in\E_K}\tau_\sigma \text{D}_{K,\sigma}w_i^\eps 
	- \eps\m(K) w^\eps_{i,K}
  = \rho\bigg(\frac{\m(K)}{\Delta t}(u^\eps_{i,K}-u_{i,K}^{k-1})
	+ \sum_{\sigma\in\E_K}\mathcal{F}^\eps_{i,K,\sigma}\bigg),
\end{equation}
for all $K\in\T$ and $i=1,\ldots,n$, where $u^\eps_{K}=(h')^{-1}(w^\eps_{K})$
and $\mathcal{F}_{i,K,\sigma}^\eps$
is defined as in \eqref{2.flux} with $u$ replaced by $u^\eps$.
Because of $u_{i,K}^\eps\in D$ we have 
$u_{i,K}^\eps>0$ for $K\in\T$ and $i=0,\ldots,n$. 

{\em Step 2: Discrete entropy inequality.}
The key step of the proof is the following lemma.

\begin{lemma}[Discrete entropy inequality]
Let the assumptions of Theorem \ref{thm.ex} hold, let $0<\rho\le 1$, $\eps>0$, and
let $u^\eps$ be a solution to \eqref{3.approx}. Then $u_{i,\sigma}^\eps>0$
for all $\sigma\in\E$, $i=0,\ldots,n$ and 
\begin{align}
  \rho {\mathcal H}[u^\eps] 
	&+ \rho c_A\Delta t\sum_{i=1}^n\sum_{\sigma\in\E}\tau_\sigma
	(u_{i,\sigma}^\eps)^{2(s-1)}(\mathrm{D}_\sigma u_i^\eps)^2 +
	\eps\Delta t\sum_{i=1}^n\|w_i^\eps\|_{1,2,\T}^2
	\le \rho{\mathcal H}[u^{k-1}]. \label{3.ei}
\end{align}
\end{lemma}

\begin{proof}
First, we prove that $u_{i,\sigma}^\eps>0$ for all $\sigma\in\E$ and $i=0,\ldots,n$.
Indeed, if $\sigma\in\E_{{\rm ext},K}$, we have $u_{i,K}^\eps=u_{i,K,\sigma}^\eps>0$
and hence $u_{i,\sigma}^\eps=u_{i,K}^\eps>0$. Thus, let $\sigma=K|L\in\Eint$.
Again, if $u_{i,K}^\eps=u_{i,L}^\eps>0$, it follows that $u_{i,\sigma}^\eps>0$.
Otherwise, $u_{i,\sigma}^\eps$ is the unique solution to
$$
  h_i''(u_{i,\sigma}^\eps) = \frac{u_{i,K}^\eps-u_{i,L}^\eps}{h_i'(u_{i,K}^\eps)
	- h_i'(u_{i,L})}>0,
$$
and we deduce from the strict monotonicity that  
$u^\eps_{i,\sigma} \ge \min\{ u^\eps_{i,K}, u^\eps_{i,L} \} >0$.

Next, multiplying \eqref{3.approx} by $\Delta tw_{i,K}^\eps$, summing over
$i=1,\ldots,n$ and $K\in\T$, and applying discrete integration by parts, we find that
\begin{align*}
	0 &= \rho\sum_{i=1}^n\sum_{K\in\T}\m(K)(u_{i,K}^\eps-u_{i,K}^{k-1})w_{i,K}^\eps
	-  \rho \Delta t \sum_{i=1}^n \sum_{\substack{\sigma\in\E_{\rm int} 
	\\ \sigma=K|L}} \mathcal{F}^\eps_{i,K,\sigma} \text{D}_{K,\sigma} w_{i}^\eps \\
	&\phantom{xx}{}+ \eps\Delta t\sum_{i=1}^n \|w_{i}^\eps\|^2_{1,2,\T} 
	= I_1+I_2+I_3.
\end{align*}
We know from Hypothesis (H4) that $h$ is convex such that, because of 
$w_{K}^\eps=h'(u^\eps_K)$,
$$
  I_1 \ge \rho\sum_{K\in\T}\m(K)\big(h(u^\eps_K)-h(u^{k-1}_K)\big)
	= \rho\big({\mathcal H}[u^\eps]-{\mathcal H}[u^{k-1}]\big).
$$
Furthermore, by Lemma \ref{lem.chain}, the symmetry of
$(H_{ij})$, and Hypothesis (H5),
\begin{align*}
  I_2 &= \rho\Delta t\sum_{i,j=1}^n \sum_{\substack{\sigma\in\E_{\rm int} 
	\\ \sigma=K|L}} \tau_\sigma A_{ij}(u_\sigma^\eps)\text{D}_{K,\sigma} u_j^\eps
	\text{D}_{K,\sigma}(h'(u^\eps))_i \\
	&= \rho\Delta t\sum_{i,j=1}^n \sum_{\substack{\sigma\in\E_{\rm int} 
	\\ \sigma=K|L}} \tau_\sigma (H(u_\sigma^\eps)\text{D}_{K,\sigma}u^\eps)_i 
	A_{ij}(u_\sigma^\eps)\text{D}_{K,\sigma} u_j^\eps \\
  &= \rho\Delta t\sum_{i,j,\ell=1}^n \sum_{\substack{\sigma\in\E_{\rm int} 
	\\ \sigma=K|L}} \tau_\sigma\text{D}_{K,\sigma}(u^\eps_\ell)H_{\ell i}(u_\sigma^\eps)
	A_{ij}(u_\sigma^\eps)\text{D}_{K,\sigma}u_j^\eps \\
	&= \rho\Delta t\sum_{j,\ell=1}^n \sum_{\substack{\sigma\in\E_{\rm int} 
	\\ \sigma=K|L}} \tau_\sigma\text{D}_{K,\sigma}u^\eps_\ell \big(H(u_\sigma^\eps)
	A(u_\sigma^\eps)\big)_{\ell j}\text{D}_{K,\sigma}u_j^\eps \\
	&\ge \rho c_A \Delta t\sum_{j=1}^n \sum_{\substack{\sigma\in\E_{\rm int} 
	\\ \sigma=K|L}} \tau_\sigma (u_{j,\sigma}^\eps)^{2(s-1)}
	(\text{D}_{K,\sigma}u_j^\eps)^2.
\end{align*}
Putting these estimates together, we conclude the proof.
\end{proof}

We continue with the topological degree argument. Choosing
$$
  R = \frac{1}{\sqrt{\eps\Delta t}} \mathcal{H}[u^{k-1}]^{1/2}+1,
$$
the previous lemma leads to
$$
  \eps\Delta t\sum_{i=1}^n\|w_i^\eps\|^2_{1,2,\T}
	\le \rho \mathcal{H}[u^{k-1}]
	\le \eps\Delta t(R-1)^2,
$$
which gives $\sum_{i=1}^n\|w_i^\eps\|^2_{1,2,\T}<R^2$.
We conclude that $w^\eps \not\in\pa Z_R$ and $\operatorname{deg}(I-F_\eps,Z_R,0)=1$. 
Thus, $F_\eps$ admits at least one fixed point.

{\em Step 3: Limit $\eps\to 0$.}
By construction of $u^\eps$, we have
$u^\eps_K = (h')^{-1}(w^\eps_K)\in D$ for $K\in\T$. Thus, 
$\|u^\eps_i\|_{0,\infty,\T} := \max_{K \in \T} |u^\eps_{i,K}|\le 1$ for 
$i=0,\ldots,n$, and there exists a subsequence
(not relabeled) such that $u_{i,K}^\eps\to u_{i,K}^k \in [0,1]$ as $\eps\to 0$ 
for $K\in\T$, $i=1,\ldots,n$ and satisfying $u^k_K=(u^k_{1,K},\ldots,u^k_{n,K})
\in \overline{D}$. Moreover, there exists a subsequence such that 
$u^\eps_{i,\sigma} \to u^k_{i,\sigma} \in [0,1]$ as $\eps \to 0$ for 
$\sigma \in \E_{\rm int}$ and $i=0,\ldots,n$, and $u_{0,\sigma}^k$ is given by \eqref{2.usigma}. In view of the bound for 
$\sqrt{\eps}w_i^\eps$, we have, again for a subsequence,
$\eps w_{i,K}^\eps\to 0$.

We show that the total mass $\int_\Omega u_i^k dx$ is positive for $i=1,\ldots,n$. 
For this, we write $w_{i,K}^{\eps,k}:=w_{i,K}^\eps$, sum over $K \in \T$ in \eqref{3.approx}, and use the local balance equations 
$\mathcal{F}^\eps_{i,K,\sigma} + \mathcal{F}^\eps_{i,L,\sigma} = 0$ and 
$\tau_\sigma \mathrm{D}_{K,\sigma} w^\eps 
+ \tau_\sigma \mathrm{D}_{L,\sigma} w^\eps = 0$ for all 
$\sigma =K|L \in \E_{\rm int}$ to deduce that, by induction,
\begin{align*}
  \sum_{K \in \T} \m(K) u^\eps_{i,K} 
	&= \sum_{K \in \T} \m(K) u^{k-1}_{i,K} 
	- \eps \Delta t \sum_{K \in \T} \m(K) w^{\eps,k}_{i,K} \\
  &= \sum_{K \in \T} \m(K) u^0_{i,K} 
	- \eps \Delta t \sum_{j=1}^k\sum_{K \in \T} \m(K) w^{\eps,k}_{i,K},
\end{align*}
for $i=1,\ldots,n$. The limit $\eps\to 0$ yields
\begin{equation}\label{3.eq.mass}
  \sum_{K \in \T} \m(K) u^k_{i,K} 
	= \sum_{K \in \T} \m(K) u^0_{i,K} 
	= \int_\Omega u^0_i(x) dx > 0 \quad \mbox{for }i=1,\ldots,n,
\end{equation}
where the positivity in the last step follows from Hypothesis (H3).

Next, we show that $0 < u^k_{i,\sigma} <1$ for all $\sigma \in \E_{\rm int}$ and 
$i=1,\ldots,n$. In view of definition \eqref{2.usigma} of $u^k_{i,\sigma}$
and $\sum_{i=1}^n u^k_{i,K} \leq 1$, it is
sufficient to show that $u^k_{i,K} > 0$ for all $K \in \T$ and $i=1,\ldots,n$.
Let $i\in\{1,\ldots,n\}$ be fixed. Assume by contradiction that $u^k_{i,K} = 0$ 
for some $K \in \T$. Then $u^k_{i,\sigma} = 0$ for $\sigma=K|L \in \E_{\rm int}$.
The entropy inequality gives
$$
  (u_{i,L}^\eps-u_{i,K}^\eps)^2 
	\le C(\Delta t,u^{k-1})(u_{i,\sigma}^\eps)^{2(1-s)},
$$
and in the limit $\eps\to 0$
\begin{equation*}
  (u_{i,L}^k)^2 = (u_{i,L}^k-u_{i,K}^k)^2 
	\le C(\Delta t,u^{k-1})(u_{i,\sigma}^k)^{2(1-s)}.
\end{equation*}
Thus, $u_{i,\sigma}^k=0$ implies that $u_{i,L}^k=0$ (here we need $s<1$).
Next, let $L'$ be a neighboring cell of $L$. By the previous argument, 
it follows that also $u_{i,L'}^k=0$. Repeating this argument for all cells in $\T$, 
we find that $u_{i,K}^k=0$ for all $K\in\T$. Consequently,  
$\sum_{K\in\T}\m(K)u_{i,K}^k=0$, which contradicts \eqref{3.eq.mass}. 
This shows in particular that $u_{0,K}^k=1-\sum_{i=1}^n u_{i,K}^k < 1$ and
hence $u^k_K \in D$ for all $K \in \T$. Moreover, 
we deduce from the definition of $u_{i,\sigma}^k$ 
that $0 < u^k_{i,\sigma}<1$ holds for all $\sigma=K|L \in \E_{\rm int}$ and 
$i=1,\ldots,n$.

Note that $u_{0,\sigma}^k=0$ may be possible, which explains the condition
$\|A_\sigma(u_\sigma)\|<\infty$ in Hypothesis (H5) whenever $u_{0,\sigma}=0$.
Thus, together with the continuity of $A_{ij,\sigma}$ on $[0,1]\times (0,1)^n$ and
the bounds on $u_{i,\sigma}^k$ for $i=1,\ldots,n$, we can pass to the
limit $\eps\to 0$ in \eqref{3.approx} and \eqref{3.ei} to finish the proof of 
Theorem \ref{thm.ex}.

\begin{remark}[Case $s=1$]\label{rem.s1}\rm
The existence of discrete solutions to scheme \eqref{2.init}--\eqref{2.usigma}
and the validity of the entropy inequality can be extended to the case $s=1$
if $\|A_\sigma(0)\|<\infty$.
This is possible since the singular term $(u_{i,\sigma}^k)^{2(s-1)}$ disappears
when $s=1$ and the proof simplifies. However, we cannot ensure in general that 
$u^k_{i,\sigma}>0$ holds for $\sigma \in \E_{\rm{int}}$, $k \geq 1$, $i=1,\ldots,n$.
For the existence of discrete solutions, the condition $\|A_\sigma(0)\|<\infty$
is crucial. The example in Section \ref{sec.tumor} shows that the modified matrix
$A_\sigma$ may contain the factor $a(u_\sigma)=\sum_{i=1}^n c_iu_{i,\sigma}$ 
for some $c_i>0$, necessary to prove the positive (semi-) definiteness
of $HA_\sigma$. In this situation, $\|A_\sigma(0)\|$ is not a number, 
and the existence of a discrete solution cannot be guaranteed.
\qed
\end{remark}

\begin{remark}[Source terms]\label{rem.source}\rm
Our analysis still holds when we include source terms of the type $f_i(u)$
on the right-hand side of \eqref{1.eq}. We need to assume that 
$f_i\in C^0(\overline{D})$ and that there exist constants $C_f>0$, $c_f\ge 0$
such that for all $u\in D$,
\begin{equation}\label{3.source}
  \sum_{i=1}^n f_i(u) (h_i'(u_i)+h_0'(u_0))\le C_f(1+h(u)) \quad \mbox{and} \quad 
	f_i(u) \geq -c_f u_i \quad \mbox{for }i=1,\ldots,n.
\end{equation}
If we assume, in addition to Hypotheses (H1)--(H5), the  condition $\Delta t<1/C_f$
on the time step size, then the statement of Theorem \ref{thm.ex} holds with
the modified entropy inequality
$$
  (1-C_f\Delta t)\mathcal{H}[u^k] 
	+ c_A\Delta t\sum_{i=1}^n\sum_{\sigma\in\E}\tau_\sigma 
	u_{i,\sigma}^{2(s-1)}(\mathrm{D}_\sigma u_i^k)^2 
	\le \mathcal{H}[u^{k-1}] + C_f\Delta t\m(\Omega).
$$
This inequality is a direct consequence of the first assumption in
\eqref{3.source}; see, e.g., the proof of 
\cite[Theorem 1]{JuZu20}. The second assumption in \eqref{3.source}
allows us to adapt the proof of the positivity of the total mass in Step 3, 
giving after an induction
$$
  \sum_{K \in \T} \m(K) u_{i,K}^k
	\ge \frac{\int_\Omega u_i^{0}(x) dx}{(1+c_f\Delta t)^k} > 0.
$$
The remaining proof is unchanged.
\qed
\end{remark}


\section{Proof of Theorem \ref{thm.conv}}\label{sec.conv}

We prove first some estimates uniform in $\Delta x$ and $\Delta t$ and then
deduce the compactness properties.

\subsection{A priori estimates}

We introduce the discrete time derivative of a function $v\in V_{\T,\Delta t}$:
$$
  \pa_t^{\Delta t} v(x,t) = \pa_t^{\Delta t} v^k(x) 
	= \frac{1}{\Delta t}(v^k(x)-v^{k-1}(x)), \quad 
	(x,t) \in \overline{\Omega} \times (t_{k-1},t_k],\ k=1,\ldots,N_T.
$$

\begin{lemma}[Uniform estimates]\label{lem.est}
Let the assumptions of Theorem \ref{thm.ex} hold.
Then there exists a constant $C>0$ independent of $\Delta x$ and $\Delta t$
such that for all $i=1,\ldots,n$,
$$
  \max_{k=1,\ldots,N_T}\|u_i^k\|_{0,1,\T}
  + \sum_{k=1}^{N_T}\Delta t\|u_i^k\|_{1,2,\T}^2 
	+ \sum_{k=1}^{N_T}\Delta t\|\pa_t^{\Delta t}u_i^k\|_{-1,2,\T}^2 \le C.
$$
\end{lemma}

\begin{proof}
We sum \eqref{2.ei} over $k=1,\ldots,j$ (with $j\le N_T$) and $i=1,\ldots,n$, and use the facts that 
$0<u_{i,\sigma}^k\le 1$ and $s< 1$ to obtain 
$$
  \mathcal{H}[u^j] + c_A\sum_{k=1}^{j}\Delta t\sum_{i=1}^n\sum_{\sigma\in\E}
	\tau_\sigma(\text{D}_\sigma u_i^k)^2
  \le \mathcal{H}[u^0].
$$
Since the entropy dominates the $L^1$ norm thanks to (H4), the previous inequality implies that
$$
  \max_{k=1,\ldots,j}\sum_{i=1}^n\|u_i^k\|_{0,1,\T}
	+ \sum_{k=1}^{j}\Delta t\sum_{i=1}^n|u_i^k|_{1,2,\T}^2\le \mathcal{H}[u^0]+c_h \m(\Omega).
$$
The discrete Poincar\'e--Wirtinger inequality \cite[Theorem 3.6]{BCF15} 
gives the existence of a constant $C$ only depending on $u^0$ and $\Omega$ 
such that $\sum_{k=1}^{j}\Delta t\|u_i^k\|_{0,2,\T}^2\le C$.

For the estimate of the discrete time derivative, 
we choose $\phi\in V_\T$ with $\|\phi\|_{1,2,\T}=1$
and multiply scheme \eqref{2.fvm} by $\phi_K$, sum over $K\in\T$, $k=1,\ldots,N_T$, 
and apply discrete integration by parts:
\begin{align*}
  \sum_{K\in\T} \m(K)\frac{u^k_{i,K}-u^{k-1}_{i,K}}{\Delta t} \phi_K 
	&= -\sum_{j=1}^n \sum_{\substack{\sigma\in\E_{{\rm int}} \\ \sigma=K|L}} 
	\tau_\sigma A_{ij}(u^k_\sigma) \text{D}_{K,\sigma} u^k_j\text{D}_{K,\sigma} \phi 
	=: J_1.
\end{align*}
The boundedness of $u_{i,\sigma}^k$ and the Cauchy-Schwarz inequality imply that
\begin{align*}
  |J_1| &\le \sum_{j=1}^n \max_\Omega|A_{ij}(u_\sigma^k)|
	|u^k_j|_{1,2,\T}|\phi|_{1,2,\T}
	\le C\sum_{j=1}^n \|u^k_j\|_{1,2,\T}\|\phi\|_{1,2,\T}.
\end{align*}
We infer that
$$
  \sum_{k=1}^{N_T}\Delta t\bigg\|\frac{u_i^k-u_i^{k-1}}{\Delta t}\bigg\|_{-1,2,\T}^2
	= \sup_{\|\phi\|_{1,2,\T}=1}\sum_{k=1}^{N_T}\Delta t\bigg|\sum_{K\in\T}\m(K)
	\frac{u_i^k-u_i^{k-1}}{\Delta t}\phi_K\bigg|^2 \le C. 
$$
This finishes the proof.
\end{proof}


\subsection{Compactness properties}

Let $(\mathcal{D}_m)_{m\in\N}$ be a sequence of admissible meshes of $\Omega_T$
satisfying the mesh regularity \eqref{2.dd} uniformly in $m\in\N$.
We claim that the estimates from Lemma \ref{lem.est}
imply the strong convergence of a subsequence of $(u_{i,m})$.

\begin{proposition}[Strong convergence]\label{prop.conv}\sloppy
Let the assumptions of Theorem \ref{thm.conv} hold and let $(u_m)_{m\in\N}$ be
a sequence of discrete solutions to \eqref{2.init}--\eqref{2.usigma} constructed
in Theorem \ref{thm.ex}. Then there exists a subsequence of $(u_m)$, which is not
relabeled, and $u=(u_1,\ldots,u_n)\in L^\infty(\Omega_T)$ such that for any
$i=1,\ldots,n$,
$$
  u_{i,m}\to u_i\quad\mbox{strongly in }L^p(\Omega_T)\mbox{ as }m\to\infty,\quad
	1 \leq p<\infty.
$$
\end{proposition}

\begin{proof}
The result follows from the discrete Aubin--Lions lemma of \cite[Theorem 3.4]{GaLa12}
if the following two properties are satisfied:
\begin{itemize}
\item[(i)] For any sequence $(v_m)_{m \in \N} \subset V_{\T_m}$ such that 
there exists $C>0$ with $\|v_m\|_{1,2,\T_m} \le C$ for all $m \in \N$, 
there exists $v \in L^2(\Omega)$ satisfying, up to a subsequence, $v_m \to v$ 
strongly in $L^2(\Omega)$.
\item[(ii)] If $v_m \to v$ strongly in $L^2(\Omega)$ and $\|v_m\|_{-1,2,\T_m} \to 0$ 
as $m \to \infty$, then $v=0$.
\end{itemize}
Property (i) follows from \cite[Lemma 5.6]{EGH08}, while property (ii) can be 
replaced by the condition that $\|\cdot\|_{1,2,\T_m}$ and $\|\cdot\|_{-1,2,\T_m}$ 
are dual norms with respect to the $L^2(\Omega)$ norm, thanks to 
\cite[Remark 6]{GaLa12}, which is the case here. 
Then \cite[Theorem 3.4]{GaLa12} implies that there exists
a subsequence (not relabeled) such that $u_{i,m}\to u_i$ strongly in 
$L^2(0,T;L^2(\Omega))$ as $m\to\infty$. 
We deduce from the $L^\infty$ bound for $(u_{i,m})$ that
$u_{i,m}\to u_i$ strongly in $L^p(\Omega_T)$ for any $1 \leq p<\infty$.
\end{proof}

\begin{lemma}[Convergence of the gradient]
Under the assumptions of Proposition \ref{prop.conv}, there exists a subsequence
of $(u_m)_{m\in\N}$ such that for $i=1,\ldots,n$,
$$
  \na^m u_{i,m}\rightharpoonup\na u_i\quad\mbox{weakly in }L^2(\Omega_T)
	\mbox{ as }m\to\infty,
$$
where $\na^m$ is defined in Section \ref{sec.main}.
\end{lemma}

\begin{proof}
Lemma \ref{lem.est} implies that $(\na^m u_{i,m})$ is bounded in $L^2(\Omega_T)$.
Thus, for a subsequence, $\na^m u_{i,m}\rightharpoonup v_i$ weakly in $\Omega_T$
as $m\to\infty$. It is shown in \cite[Lemma 4.4]{CLP03} that $v_i=\na u_i$.
\end{proof}


\subsection{Convergence of the scheme}

We show that the limit $u$ from Proposition \ref{prop.conv} is a weak solution
to \eqref{1.eq}--\eqref{1.bic}. Let $i\in\{1,\ldots,n\}$ be fixed, let
$\psi_i\in C_0^\infty(\Omega\times[0,T))$ be given, and let
$\eta_m:=\max\{\Delta x_,\Delta t_m\}$ be sufficiently small such that
$\operatorname{supp}(\psi_i)\subset\{x\in\Omega:\dist(x,\pa\Omega)>\eta_m\}
\times[0,T)$. Furthermore, let $\psi_{i,K}^k:=\psi_i(x_K,t_k)$. We multiply scheme
\eqref{2.fvm} by $\Delta t_m\psi_{i,K}^{k-1}$ and sum over $K\in\T_m$ and
$k=1,\ldots,N_T$. Then $F_1^m+F_2^m=0$, where
\begin{align*}
	& F_1^m = \sum_{k=1}^{N_T}\sum_{K\in\T_m}\m(K)\big(u_{i,K}^k-u_{i,K}^{k-1}\big)
	\psi_{i,K}^{k-1}, \\
	& F_2^m = -\sum_{j=1}^n\sum_{k=1}^{N_T}\Delta t_m\sum_{K\in\T_m}
	\sum_{\sigma\in\Eint}\tau_\sigma 
	A_{\sigma,ij}(u^k_\sigma) \text{D}_{K,\sigma} u^k_j \psi_{i,K}^{k-1}.
\end{align*}
Furthermore, we introduce
\begin{align*}
	F_{10}^m &= -\int_0^T\int_{\Omega}u_{i,m} \pa_t\psi_i dxdt
	- \int_{\Omega}u_{i,m}(x,0)\psi_i(x,0)dx, \\
	F_{20}^m &= \sum_{j=1}^n \int_0^T\int_{\Omega} A_{ij}(u_{m}) \na^m u_{j,m}\cdot 
	\na\psi_i dxdt.
\end{align*}
It follows from the convergence results from the previous subsection, the continuity 
of $A_{ij}$, and the assumption on the initial data that, as $m\to\infty$,
\begin{align*}
  F_{10}^m+F_{20}^m
	&\to -\int_0^T\int_\Omega u_i\pa_t\psi_i dxdt
	- \int_\Omega u_i^0(x) \psi_i(x,0)dx 
	+ \sum_{j=1}^n \int_0^T\int_\Omega A_{ij}(u) \na u_j \cdot\na \psi_i dxdt. 
\end{align*}
We prove that $F_{j0}^m-F_j^m\to 0$ as $m\to\infty$ for $j=1,2$, since this
shows that $F_{10}^m+F_{20}^m\to 0$, finishing the proof.

We start with the first difference $F_{10}^m-F_1^m$. It is shown in 
\cite[Theorem 5.2]{CLP03}, using the $L^\infty(\Omega_T)$ bound
for $u_{i,m}$ and the regularity of $\phi$, that $F_{10}^m-F_1^m\to 0$.
It remains to verify that $|F_{20}^m-F_2^m|\to 0$. To this end, we apply
discrete integration by parts and write $F_2^m=F_{21}^m+F_{22}^m$, where
\begin{align*}
  F_{21}^m &= \sum_{j=1}^n \sum_{k=1}^{N_T}\Delta t_m\sum_{K\in\T_m}
	\sum_{\sigma\in\Eint} 
	\tau_\sigma A_{ij}(u^k_K)\text{D}_{K,\sigma} u^k_{j}\text{D}_{K,\sigma} 
	\psi_{i}^{k-1}, \\
  F_{22}^m &= \sum_{j=1}^n\sum_{k=1}^{N_T}\Delta t_m\sum_{K\in\T_m}
	\sum_{\sigma\in\Eint}\tau_\sigma \big(A_{\sigma,ij}(u^k_\sigma)-A_{\sigma,ij}
	(u^k_{0,K},u^k_K)\big) \text{D}_{K,\sigma} u_j^k\text{D}_{K,\sigma} \psi_{i}^{k-1}.
\end{align*}
Here, we used the equality $A_\sigma(u^k_{0,K},u^k_K)=A(u^k_K)$ since $u^k_K \in D$ 
and $u^k_{0,K}=1-\sum_{i=1}^n u^k_{i,K}$ for all $K \in \T$, coming from 
Hypothesis (H5). The definition of the discrete gradient $\na^m$ in 
Section \ref{ssec.main} gives
\begin{align*}
  |F_{20}^m - F_{21}^m| 
	&\le \sum_{j=1}^n \sum_{k=1}^{N_T} \sum_{K\in\T_m} \sum_{\sigma\in\Eint} \m(\sigma) 
	|A_{ij}(u^k_K)| |\text{D}_{K,\sigma} u^k_j| \\ 
	&\phantom{xx}{}\times \bigg|\int_{t_{k-1}}^{t_k}\bigg(\frac{\text{D}_{K,\sigma} 
	\psi_i^{k-1}}{\dist_{\sigma}} - \frac{1}{\m(T_{K,\sigma})} \int_{T_{K,\sigma}} 
	\na\psi_i \cdot \nu_{K,\sigma} dx \bigg) dt \bigg|.
\end{align*}
It is shown in the proof of \cite[Theorem 5.1]{CLP03} that there exists a constant 
$C_0>0$ such that
$$
  \bigg|\int_{t_{k-1}}^{t_k}\bigg(\frac{\text{D}_{K,\sigma} 
	\psi_i^{k-1}}{\dist_{\sigma}} - \frac{1}{\m(T_{K,\sigma})} \int_{T_{K,\sigma}} 
	\na\psi_i \cdot \nu_{K,\sigma} dx \bigg) dt \bigg|
	\le C_0\Delta t_m\eta_m.
$$
Hence, by the uniform $L^\infty$ bound for $u^k$ and the Cauchy--Schwarz inequality,
\begin{align*}
  |F_{20}^m-F_{21}^m| &\le C_0\eta_m \sum_{j=1}^n \sum_{k=1}^{N_T} \Delta t_m 
	\sum_{K\in\T_m} \sum_{\sigma\in\Eint}\m(\sigma)|A_{ij}(u^k_K)| \, 
	|\text{D}_{K,\sigma} u^k_j | \\
  &\le C\eta_m\sum_{j=1}^n \sum_{k=1}^{N_T} \Delta t_m \|u_j^k\|_{1,2,\T_m} 
	\bigg(\sum_{K\in\T_m}\sum_{\sigma\in\Eint}\m(\sigma)\dist_\sigma 
	\bigg)^{1/2}.
\end{align*}
We deduce from the mesh regularity \eqref{2.dd} and the assumption 
$\Omega \subset \R^2$ that
$$
  \sum_{K\in\T_m}\sum_{\sigma\in\Eint}\m(\sigma)\dist_\sigma
	\le \frac{1}{\zeta}\sum_{K\in\T_m}\sum_{\sigma\in\Eint}\m(\sigma)\dist(x_K,\sigma)
	\le \frac{2}{\zeta}\sum_{K\in\T_m}\m(K) = \frac{2}{\zeta}\m(\Omega).
$$
Hence, we have
$$
  |F_{20}^m-F_{21}^m| \le C\eta_m\sum_{j=1}^n \sum_{k=1}^{N_T} 
	\Delta t_m \|u_j^k\|_{1,2,\T_m} \le C\eta_m\to 0.
$$

For the estimate of $F_{22}^m$, we need the Lipschitz continuity of
$A_{\sigma,ij}$, the upper bound $u_{\ell,\sigma}^k\le (u_{\ell,K}^k+u_{\ell,L}^k)/2$ 
for $\sigma\in\E_{\rm int}$ and $\ell=0,\ldots,n$, the equality 
$u^k_{0,K}=1-\sum_{\ell=1}^n u^k_{\ell,K}$ for all $K \in \T$, and the 
Cauchy--Schwarz inequality:
\begin{align*}
  |F_{22}^m|&\le C\eta_m\|\psi_i\|_{C^1(\overline{\Omega}_T)}G_m, \quad\mbox{where} \\
  G_m &= \sum_{j=1}^n \sum_{k=1}^{N_T}\Delta t_m \sum_{K\in\T_m} 
	\sum_{\sigma\in\Eint}\tau_\sigma |A_{\sigma,ij}(u^k_\sigma)
	-A_{\sigma,ij}(u_{0,K}^k,u^k_K)|\text{D}_{\sigma} u_j^k \\
	&\le C \sum_{j=1}^n \sum_{k=1}^{N_T}\Delta t_m \sum_{K\in\T_m} 
	\sum_{\sigma\in\Eint}\tau_\sigma \bigg(\sum_{\ell=0}^n 
	|u_{\ell,\sigma}^k-u_{\ell,K}^k| \bigg)\text{D}_{\sigma} u_j^k \\
	&\le \frac{C}{2}\sum_{j=1}^n \sum_{k=1}^{N_T}\Delta t_m \sum_{K\in\T_m} 
	\sum_{\sigma\in\Eint}\tau_\sigma \bigg(\sum_{\ell=0}^n 
	|u_{\ell,L}^k-u_{\ell,K}^k| \bigg)\text{D}_{\sigma}u_j^k \\
	&\le C \sum_{\ell,j=1}^n \sum_{k=1}^{N_T}\Delta t_m \sum_{K\in\T_m} 
	\sum_{\sigma\in\Eint}\tau_\sigma \text{D}_{\sigma} u_{\ell}^k
	\text{D}_{\sigma} u_j^k \\
	&\le C\bigg(\sum_{\ell=1}^n\sum_{k=1}^{N_T}\Delta t_m\sum_{\sigma\in\E}\tau_\sigma
	(\text{D}_\sigma u_\ell^k)^2\bigg)^{1/2}\bigg(\sum_{j=1}^n\sum_{k=1}^{N_T}
	\Delta t_m\sum_{\sigma\in\E}\tau_\sigma(\text{D}_\sigma u_j^k)^2\bigg)^{1/2}.
\end{align*}
By Lemma \ref{lem.est}, the right-hand side is bounded uniformly in $m$.
Thus, $|F_{22}^m|\le C\eta_m\to 0$ and $|F_{20}^m-F_2^m|\le |F_{20}^m-F_{21}^m|
+ |F_{22}^m|\to 0$ as $m\to\infty$. This finishes the proof.

\begin{remark}\label{rem.conv}\rm
Let us mention some possible extensions of Theorem \ref{thm.conv}.
We can easily adapt the proof to the case $s=1$ if we assume in addition to
(H5) that $\|A_\sigma(0)\|<\infty$ holds. Moreover, we can include source terms
$f_i\in C^0(\overline{D})$ for $i=1,\ldots,n$ such that the conditions 
\eqref{3.source} and $\Delta t_m < 1/C_f$ for $m \in \N$ are fulfilled. 
Then, following the proof in \cite[Theorem 2]{JuZu20}, we can show that 
Theorem \ref{thm.conv} still holds.
\qed
\end{remark}


\section{Numerical examples}\label{sec.num}

We present in this section some numerical experiments in one and two space dimensions. 

\subsection{Implementation of the scheme}

The finite-volume scheme \eqref{2.init}--\eqref{2.usigma} is implemented in MATLAB. 
Since the numerical scheme is implicit in time, we have to solve a nonlinear system 
of equations at each time step.
In the one-dimensional case, we use Newton's method. Starting from 
$u^{k-1}=(u^{k-1}_1, u^{k-1}_2)$, we apply a Newton method with precision 
$\eps = 10^{-10}$ to approximate the solution to the scheme at time step $k$.
In the two-dimensional case, we use a Newton method complemented by an adaptive 
time-stepping strategy to approximate the solution of the scheme at time $t_k$. 
More precisely, starting again from $u^{k-1}=(u^{k-1}_1, u^{k-1}_2)$, we launch a 
Newton method. If the method does not converge with precision 
$\eps= 10^{-10}$ after at most $50$ steps, we multiply the time step by a factor $0.2$ and restart the 
Newton method. At the beginning of each time step, we increase the value of the previous time 
step size by multiplying it by $1.1$. Moreover, we impose the condition 
$10^{-8}\leq \Delta t_{k} \leq 10^{-2}$ with 
an initial time step size set to $ 10^{-5}$.

\subsection{Test case 1: Rate of convergence in space}\label{ssec.convergence}

In this section, we illustrate the order of convergence in space for the 
Maxwell--Stefan model presented in Section \ref{ssec.MaxwellStefan} in one 
space dimension with $\Omega = (0,1)$. We consider a similar test case as in 
\cite[Section 6.2]{JuLe19} but with a discontinuous initial datum $u_1^0$. 
We choose the coefficients $d_0=1/0.168$, $d_1=1/0.68$, and $d_2= 1/ 0.883$ and impose 
the initial datum
\begin{align*}
  & u_1^0(x) = 0.8 \cdot\mathbf{1}_{(0,0.5)}(x),\quad u_2^0(x) = 0.2. 
\end{align*}
Since exact solutions to the Maxwell-Stefan model are not explicitly known, we compute 
a reference solution on a uniform mesh composed of $5120$ cells and with time step 
size $\Delta t = (1/5120)^2$. We use this rather small value of $\Delta t$,
because the Euler discretization in time exhibits a first-order convergence rate, 
while we expect, as observed for instance in \cite{CaGa20}, a second-order convergence 
rate in space for scheme \eqref{2.init}--\eqref{2.usigma}, due to the logarithmic 
mean used to approximate the mobility coefficients in the numerical fluxes. 
We compute approximate solutions on uniform meshes made of $40$, $80$, $160$, 
$320$, $640$, and $1280$ cells, respectively. In Figure \ref{Fig.conv}, we present
the sum of the $L^1(\Omega)$ norms of the differences between the 
approximate solution $u_i$ and
the average of the reference solution $u_{i, \rm ref}$ at the final time $T=10^{-2}$. 
As expected, we observe a second-order convergence rate in space as in \cite{JuLe19}.

\begin{figure}
\begin{center}
\includegraphics[scale=1.1]{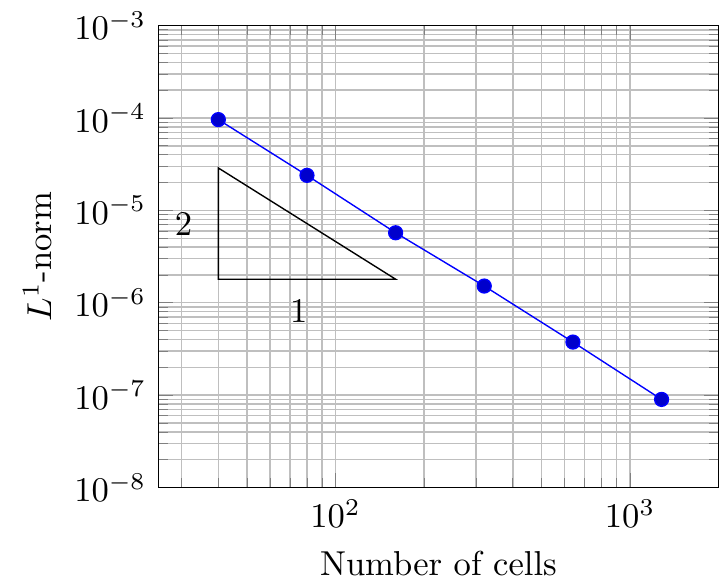}
\end{center}
\caption{Test case 1: $L^1$ norm of the error in space at final time $T=10^{-2}$.}
\label{Fig.conv}
\end{figure}

\subsection{Test case 2: Long-time behavior}\label{ssec.longtime}

Following \cite[Section 5.3]{CaGa20}, we study the long-time behavior of the scheme for the cross-diffusion system for thin-film solar cells, introduced in 
Section \ref{sec.exam.thinfilm}, with reactions terms. More precisely, 
for $\Omega = (0,1)^2$ and the final time $T=15$, we consider the system
\begin{align}\label{4.thinfilm1}
  \pa_t u_1 - \diver\big(A_{11}(u)\na u_1 + A_{12}(u)\na u_2\big) 
  &= r_1(u),\\
  \pa_t u_2 - \diver\big(A_{21}(u)\na u_1 + A_{22}(u)\na u_2\big) 
  &= -2 r_1(u),\label{4.thinfilm2}
\end{align} 
where
$$
  r_1(u) = (u_2^+)^2 - 1000 u^+_1 (1-u_1-u_2)^+,
$$
and the coefficients of the diffusion matrix $A(u)$ are given by \eqref{2.diff.matrix.thinfilm}. We choose, similar to \cite{CaGa20}, $a_{10} =  1$, 
$a_{20} = 0.1$, $a_{12} = 0$, and $a_{21} = 0$. Observe that these coefficients 
do not satisfy the assumptions of Lemma \ref{lem.thinfilm}. Finally, we impose 
the initial datum
$$
  u_1^0(x,y) = \frac{9}{44} \frac{\mathbf{1}_{(0,0.5)^2}(x,y)}{0.5^2},\quad
  u_2^0(x,y) = \frac{2}{11} \frac{\mathbf{1}_{(0.5,1)^2}(x,y)}{0.5^2}.
$$
The steady state of this system is given by
\begin{align*}
u^\infty_1 = \frac{9}{44} - \alpha, \quad u^\infty_2 = \frac{2}{11}+2\alpha,
\end{align*}
where $\alpha = (-5 \sqrt{206530}+4504)/10956$ is the unique root of the polynomial of 
degree two given by $r_1(u^\infty)$ which ensures the nonnegative of the steady state $u^\infty=(u^\infty_1,u^\infty_2)$, see \cite[Section 5.3]{CaGa20} for more details. 

Let us notice that the source terms in \eqref{4.thinfilm1}--\eqref{4.thinfilm2} do not satisfy the assumptions \eqref{3.source} of Remark \ref{rem.source}. Indeed, in this case, we consider the (discrete) relative Boltzmann entropy
\begin{align}\label{4.entropy.rel}
  \mathcal{H}[u|u^\infty] = \sum_{i=0}^2 \sum_{K \in \T} \m(K) \bigg( u_{i,K} \log\frac{u_{i,K}}{u^\infty_i} + u^\infty_i - u_{i,K} \bigg),
\end{align}
with $u^\infty_0=1-u^\infty_1-u^\infty_2$. Now, since by construction of $u^\infty$ we have $\log(u^\infty_1 u^\infty_0) - 2\log(u^\infty_2) = \log(1000)$, we conclude that
\begin{align*}
r_1(u) (h'_1(u_1)+h'_0(u_0)) -2r_1(u) (h'_2(u_2)+h'_0(u_0))= r_1(u) \left( \log(1000 u_1 u_0) - \log(u_2^2) \right) \leq 0.
\end{align*}
In particular, under the assumptions of Lemma \ref{lem.thinfilm} and adapting the proof of Theorem \ref{thm.ex}, the following discrete entropy inequality holds:
\begin{equation*}
  \mathcal{H}[u^k|u^\infty] + \alpha \Delta t\sum_{i=1}^n\sum_{\sigma\in\E_{\rm{int}}}\tau_\sigma 
	\big(\mathrm{D}_\sigma (u_i^k)^{1/2}\big)^2	\le \mathcal{H}[u^{k-1}|u^\infty],
\end{equation*}
with $\alpha = \min_{i,j=1,\ldots,2} a_{ij}$. Then, arguing as in \cite{CCHK20}, there exist constants $\kappa>0$ 
(depending on $u^0$) and $\lambda>0$ (depending on $\alpha$, $u^0$, and $\zeta$) 
such that
$$
  \sum_{i=1}^n \|u_i^k-u^\infty_i\|^2_{0,1,\T} 
  \le \kappa \mathcal{H}[u^0|u^\infty]e^{-\lambda t_k}
  \quad\mbox{for all }k\ge 1.
$$

Figure \ref{Fig.longtime} illustrates, in semilogarithmic scale and for a mesh of 
$\Omega$ made of
3584 triangles (see \cite[Fig. 2 left]{JuZu20}), the temporal evolution of the 
discrete relative entropy \eqref{4.entropy.rel}. As in \cite{CaGa20}, we observe an 
exponential decay towards the steady state, 
although the assumptions of Lemma \ref{lem.thinfilm} are not fulfilled.

\begin{figure}
\begin{center}
\includegraphics[scale=1.1]{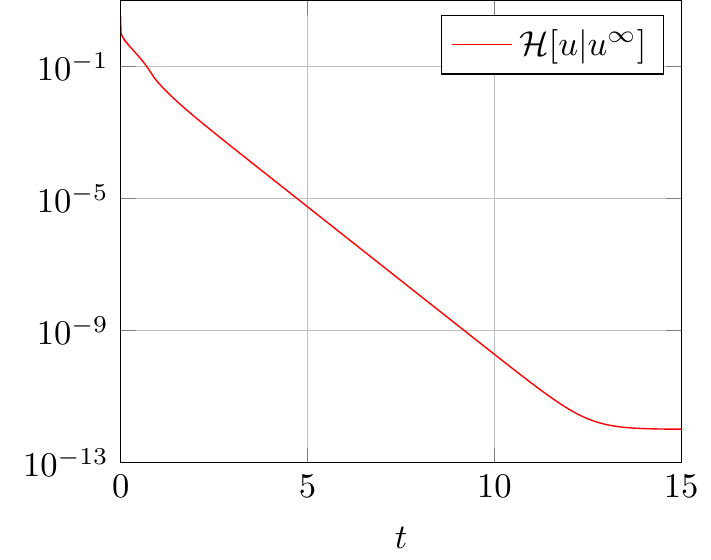}
\end{center}
\caption{Test case 2: Discrete relative entropy versus time.}
\label{Fig.longtime}
\end{figure}


\end{document}